\documentclass[12pt]{amsart}

\usepackage[text={425pt,650pt},centering]{geometry}
\usepackage{color}
\usepackage{esint,amssymb}
\usepackage{graphicx}
\usepackage{tikz}

\newtheorem{theorem}{Theorem}
\newtheorem{proposition}[theorem]{Proposition}
\newtheorem{lemma}[theorem]{Lemma}
\newtheorem{corollary}[theorem]{Corollary}

\theoremstyle{definition}
\newtheorem{remark}[theorem]{Remark}

\newcommand{\cref}[1]{Corollary~\ref{c.#1}}

\numberwithin{equation}{section}
\numberwithin{theorem}{section}

\newcommand{\R}{\mathbb{R}}

\newcommand{\ep}{\varepsilon}

\newcommand{\RR}{\boldsymbol{R}}
\newcommand{\rrho}{{\boldsymbol{\rho}}}
\newcommand{\UU}{{\boldsymbol{U}}}

\newcommand{\average}{{\mathchoice {\kern1ex\vcenter{\hrule height.4pt
width 6pt depth0pt} \kern-9.7pt} {\kern1ex\vcenter{\hrule
height.4pt width 4.3pt depth0pt} \kern-7pt} {} {} }}
\newcommand{\ave}{\average\int}

\newcommand{\ltt}{ {\textstyle \frac{\lambda^t}{t}} }
\newcommand{\s}{s}

\renewcommand{\bar}{\overline}
\renewcommand{\tilde}{\widetilde}

\begin{document}

\title[Quantitative stability of the free boundary in the obstacle problem]{Quantitative stability of the free boundary in the obstacle problem}

\begin{abstract}
We prove some detailed quantitative stability results for the contact set and the solution of the classical obstacle problem  in $\R^n$ ($n \ge 2$) under perturbations of the obstacle function, which is also equivalent to studying the variation of the equilibrium measure in classical potential theory under a perturbation of the external field. 

To do so, working in the setting of the whole space, we examine the evolution of the free boundary $\Gamma^t$ corresponding to the boundary of the contact set for a family  of obstacle functions $h^t$.  
Assuming that $h=h^t (x) = h(t,x)$ is $C^{k+1,\alpha}$ in $[-1,1]\times \R^n$ and that the initial free boundary $\Gamma^0$  is regular, we prove that $\Gamma^t$ is twice differentiable in $t$ in a small neighborhood of $t=0$. Moreover, we show that the ``normal velocity'' and the ``normal acceleration'' of $\Gamma^t$ are respectively $C^{k-1,\alpha}$ and $C^{k-2,\alpha}$ scalar fields on $\Gamma^t$. 
This is accomplished by deriving equations for these velocity and acceleration  and studying the regularity of their solutions via single and double layers estimates from potential theory.
\end{abstract}

\author[S. Serfaty]{Sylvia Serfaty}
\address{Courant Institute, New York University, 251 Mercer st, New York, NY 10012, USA.}
\email{serfaty@cims.nyu.edu}

\author[J. Serra]{Joaquim Serra}
\address{ETH Z\"urich, Department of Mathematics, R\"amistrasse 101, 8092 Z\"urich, Switzerland.}
\email{joaquim.serra@math.ethz.ch}

\date{\today}

\maketitle

\section{Introduction}

\subsection{Motivation of the problem}

Consider the classical obstacle problem (see for instance \cite{KN,Caf2}). If the obstacle $h$ is perturbed into $h+t \xi$ with $t$ small and $\xi$ regular enough, how much does the contact set (or coincidence set) move? The best known answer to this question is in a paper by Blank \cite{Blank}  which proves that the new contact set is $O(t)$-close to the old one in Hausdorff distance, in the setting of a bounded domain with Dirichlet boundary condition.  Some results are also proved in \cite{schaeffer} in an analytic setting, by Nash-Moser inversion.

Our paper is concerned with getting stronger and more quantitative stability estimates, in particular obtaining closeness of the contact sets in $C^{k, \alpha}$ norms with explicitly described first and second derivatives with respect to $t$, which come together with an explicit asymptotic expansion of the solution itself.
We believe that such results are of natural and independent interest for the obstacle problem. They are also for us motivated by  an application on the analysis of Coulomb systems in statistical mechanics for the paper \cite{ls} which relies on the present paper. 

Let us get into more detail on this aspect.
  In potential theory, the  so-called (Frostman) ``equilibrium measure" for Coulomb interactions with an external ``field" $Q$  is the unique probability measure $\mu$  on $\R^n$ which minimizes
    \begin{equation}
  \int_{\R^n \times \R^n} P(x-y) d\mu(x) \, d\mu(y)+ \int_{\R^n} Q(x) d\mu(x)\end{equation}
  where $P$ is the Newtonian potential in dimension $n$.
 If $Q$ grows fast enough at infinity, then setting 
  \begin{equation}\label{em}u(x)=\int_{\R^n} P(x-y) d\mu(y),\end{equation}
  the equilibrium measure $\mu$ is compactly supported and uniquely characterized by the fact that there exists a constant $c$ such that 
  $$ u\ge c- \frac{Q}{2} \quad \text{and} \ u= c- \frac{Q}{2} \quad \mu-a.e,$$ cf. for instance \cite{safftotik}.
  We thus find that $\mu= -\Delta u$ where
 $u$ solves the classical obstacle problem in the whole space
 $$\min \{- \Delta u, u-h\}=0$$
  with obstacle $h= c-\frac{Q}{2}$ --- the two problems (identifying the equilibrium measure and solving the obstacle problem) are in fact convex dual minimization problems as seen in   \cite{ET}, cf. for instance  \cite[Chap. 2]{zurichlectures} for a description of this correspondence. Thus, the support of the equilibrium measure is equal to the contact set  wherever the obstacle is ``active". 
  
   The understanding of the dependence of the equilibrium measure on the external field -- which is thus equivalent to the understanding of the dependence of the solution and its contact set on the obstacle function -- is crucial for the analysis of   systems of particles with logarithmic or Coulomb interactions, in particular it allows to show that the linear statistics  of fluctuations of such systems converge to Gaussians. Following the method first introduced by \cite{johansson}, this relies on the computation of the Laplace transform of the fluctuations, which directly leads to considering  the same system but with perturbed external field.  Previously, the analysis of the perturbation of the equilibrium measure, as done in \cite{ahm}, were relying on Sakai's theory \cite{sakai}, a  complex analytic approach which is thus only valid in two dimensions and  imposed placing analyticity assumptions on the external field and the boundary of the coincidence set. 
  
In that context,  the evolution  of the contact sets sometimes goes by the name ``Laplacian growth" or ``Hele-Shaw  flow" or ``Hele-Shaw equation", cf \cite{am1,am2}, and seems related to the quantum Hele-Shaw flow introduced by the physicists Wiegmann and Zabrodin \cite{W}. It has only been examined in dimension~2.

 \subsection{Setting of the study}

Both for simplicity and for the applications we have in mind mentioned above, we consider global solutions to the obstacle problem in $\R^n$, $n\ge2$. 
We note that the setting in $\R^2$ is slightly different than the setting in $\R^n$ for $n\ge 3$ due to the fact that the logaritmic Newtonian potential does not decay to zero at infinity, and this will lead us to often making parallel statements about the two.
 We also note that the potential $u$ associated to the equilibrium measure in \eqref{em} behaves like $P$ at infinity, since $\mu$ is a compactly supported probability measure, i.e. tends to $0$ if $n\ge 3$ and behaves like $- \frac{1}{2\pi} \log |x|$ if $n=2$. Specifying the total  mass of $-\Delta u$ is equivalent to specifying the ratio of $\frac{u}{-\log |x|}$ at infinity in dimension $2$, or to adding an appropriate\footnote{Let $u^t$ be defined as  \eqref{eq-nge3}. For $n\ge 3$ there is a nonlinear (but monotone and continous) relation between the mass $\int_{\R^n} \Delta u^t$ and value of the constant $c^t$. For $c^t$ large enough the mass is $0$ when decreasing $c^t$  the mass increases continuously. This allows to solve the equation with prescribed mass by varying the constant $c^t$} constant to $u$ in dimension $ n\ge 3$. 

With the above motivation, in order to consider the perturbations of the obstacle,  we thus consider  for each $t\in[-1,1]$,  given  $c^t$ a function of $t$,  the function $u^t$ solving the obstacle problem
\begin{equation}\label{eq-nge3}
\min \{ -\Delta u^t , u^t-h^t \} = 0 \quad \mbox{in }\R^n,
\qquad 
\begin{cases}
\lim_{|x|\to \infty} u^t(x)  = c^t \quad & n\ge 3 
\\
\lim_{|x|\to \infty} \frac{u^t(x)}{-\log |x|}  = c^t \quad & n=2.
\end{cases}
\end{equation}

We assume that $\Delta h^0<0$  on $\{u^0=h^0\}$, i.e. the obstacle must be ``active" in the contact set, and 
\begin{equation}\label{limhinfty}
\begin{cases}
\lim_{|x|\to \infty} h^t(x)  < c^t \quad & n\ge 3 
\\
\lim_{|x|\to \infty} \frac{h^t(x)}{-\log |x|}  < c^t \quad & n=2,
\end{cases}
\end{equation}
\begin{equation}\label{rghtt}
h = h^t(x)=h(t,x) \in C^{k-1,\alpha}([-1,1]\times B_{\RR} )
\end{equation}
while
\begin{equation}\label{ct}
c = c^t = c(t) \in C^{2}([-1,1]).
\end{equation}
For $n=2$ we assume that $c>0$.

In addition, we assume that
\begin{equation} \label{compactsupp}
\Delta(h^t-h^0) \quad   \text{is compactly supported in }  B_{\RR}
\end{equation}
and
\begin{equation} \label{condinfty}
h^t-h^0\to 0 \quad \text{as}\,  |x|\to \infty, \quad \text{resp.} \  \frac{h^t-h^0}{-\log |x|} \to 0\ \quad \text{as} \, |x|\to \infty.
\end{equation}
In particular, letting $\dot{}$ denote the derivative with respect to $t$, this implies that 
\begin{equation}\label{condinfty2}
\dot h^t \to 0 \quad \text{as}\,  |x|\to \infty, \quad \text{resp.} \  \frac{\dot h^t}{-\log |x|} \to 0\ \quad \text{as} \, |x|\to \infty.
\end{equation}

Let us denote 
\[ \Omega^t := \{u^t-h^t>0\}\qquad \mbox{and} \qquad \Gamma^t := \partial{\Omega^t}\] the complement of the contact set and the free boundary, respectively.

We will assume that  all points of the ``initial'' free boundary $\Gamma^0$ are regular points in the sense of Caffarelli (see \cite{Caf1, Caf2}). 
In particular we assume that $\Omega^0$ is an open set with smooth boundary.

For the analysis of the paper it is convenient to identify precisely the quantities on which the (constants in the) estimates depend. 
To this aim, let us fix $\rrho>0$ and make the following quantitative  assumptions.

First, we assume that, for some $\UU\subset B_{\RR}$ we have 
\begin{equation}\label{rho1}
 \Delta h^0\le-\rrho \quad \mbox{in }\overline{\UU}  \qquad \mbox{and}\qquad  
 \begin{cases}
 u^0-h^0\ge \rrho \quad \mbox{in }\R^n\setminus \UU \quad & n\ge 3
 \\
 \frac{u^0-h^0}{-\log |x|}\ge \rrho\quad \mbox{in }\R^n\setminus \UU \quad & n= 2,
\end{cases}
\end{equation} 
where $U\subset B_{\RR}$ is some open set containing $\{u^0=0\}$. 

Second, we assume that
\begin{equation}\label{rho2}
\mbox{all points of } \Gamma^0 \mbox{ can be touched from both sides by balls of radius }\rrho.
\end{equation} 
This is a quantitative version of the assumption that all points of $\Gamma^0$ are regular points.

Throughout the paper, if $\mathcal C$ is a set of parameters of the problem, we denote by
$C(\mathcal C)$ a constant depending only on  $\mathcal C$.
We denote 
\begin{equation} \label{constants}
\boldsymbol{\mathcal C} :=  \big\{ n, k, \alpha, \RR, \UU, \rrho, \|h\|_{C^{k+1,\alpha}([-1,1]\times\overline{U})}, \|c\|_{C^2([-1,1])} \big\}
\end{equation}
and
\begin{equation} \label{constants2}
\boldsymbol{\mathcal C^0} :=  \big\{ n, k, \alpha, \RR, \UU, \rrho, \|h^0\|_{C^{k+1,\alpha}(\overline{U})}, c^0 \big\}
\end{equation}

For $n=2$ we also add to $\boldsymbol{\mathcal C}$ the constant $\inf_{[-1,1]} c >0$.

\subsection{Main result}

Let $t_\circ>0$, $\Psi = \Psi^t(x) = \Psi(t,x)$ be a $1$-parameter family of diffeomorphisms $\Psi: (-t_\circ,t_\circ)\times \R^n \rightarrow \R^n$. We say that $\Psi$ fixes the complement of $\UU$  if $\Psi(x)=x$ for all $x\in \R^n\setminus U$.

We say that $\Psi$ is continuously differentiable if for all $t\in (-t_\circ,t_\circ)$  there exists $\dot \Psi^t\in C^{0}(\R^n; \R^n)$ such that 
\[
\big\| \Psi^{t+s}(x) -  \Psi^{t}(x) - s \,\dot\Psi^t (x)  \big\|_{C^{0}(\R^n; \R^n)} = o(s)
\]
and 
\[
\big\| \dot\Psi^{t+s}(x) -\dot\Psi^t (x)  \big\|_{C^{0}(\R^n; \R^n)} = o(1)
\]
as $s\to 0$.

We say that $\Psi$ is twice continuously differentiable if, in addition,  for all $t\in (-t_\circ,t_\circ)$ there exists $\ddot \Psi^t\in C^{0}(\R^n; \R^n)$ 
\[
\bigg\| \Psi^{t+s}(x) -  \Psi^{t}(x) - s\,\dot\Psi^t (x)  -\frac 1 2  s ^2 \,\ddot\Psi^t (x)  \bigg\|_{C^0(\R^n; \R^n)} = o(s^2)
\]
and
\[
\big\| \ddot\Psi^{t+s}(x) -\ddot\Psi^t (x)  \big\|_{C^{0}(\R^n; \R^n)} = o(1)
\]
as $s\to 0$.

Throughout the paper, given a function  $f:(-t_\circ, t_\circ)\times Y\to \R$  we use the notation  $f= f^t(x) = f(t,x)$  and
\[
\delta_t f^s :=  \frac{f^{s+t} -f^s}{t} \qquad \mbox{and} \qquad \dot f^s :=  \lim_{t\downarrow 0}\delta_t f^s = \partial_t f(s,y).
\]

The main result of the paper is the following. In its statement, and throughout the paper, we denote by 
\[
\nu^t:\Gamma^t \to S^{n-1}
\] 
the unit normal vector to $\Gamma^t$ pointing towards $\Omega^t$.

\begin{theorem}\label{thm1}
Let $n\ge 2$, $k\ge1$, $\alpha\in(0,1)$,  and $u^t$ satisfying \eqref{eq-nge3} with $h$ and $c$ satisfying \eqref{limhinfty} ---\eqref{condinfty}.
Assume that \eqref{rho1} and \eqref{rho2} hold. 

Then,  there exists $t_\circ>0$ and a 1-parameter differentiable family of diffeomorphisms $\Psi^t \in C^{k,\alpha}\big(\R^n;\R^n\big)$ that  fixes the complement of $\UU$  and which  satisfies, for every $t\in(-t_\circ,t_\circ)$
\[
\Psi^t(\Omega^0) =\Omega^t, \qquad \Psi^t(\Gamma^0) =\Gamma^t
\]
\begin{equation}  \label{psidot}
\| \dot \Psi^t \|_{C^{k-1,\alpha}(\R^n) } \le C \qquad \mbox{and}\qquad     (\dot\Psi^t\circ (\Psi^t)^{-1})  \cdot \nu^t = \frac{\partial_{\nu^t}V^t}{\Delta h^t} \quad \mbox{on }\Gamma^t,
\end{equation}
where $V^t := \dot u ^t - \dot h^t$ is the solution\footnote{Note that since we assumed that  $\dot h^t$  tends to $0$ (resp. is $\ll |\log |x||$) at $\infty$,
$V^t$ is the unique such that $V^t + \dot h^t$ is bounded, coincides with $h^t$ in the complement of $\Omega^t$ and is harmonic in $\Omega^t$. In fact, $V^t + \dot h^t$   is the unique bounded harmonic extension of $\dot h^t$ outside of $(\Omega^t)^c$} to 
\begin{equation}\label{diriV}
\begin{cases}
\Delta  V^t  = -\Delta\dot h^t \quad &\mbox{in }\Omega^t
\\
V^t =  0 &\mbox{on }\Gamma^t
\\
\lim_{x\to \infty } V^t(x) = \dot c^t, & \quad \text{resp. } \lim_{x\to \infty } \frac{V^t(x)}{-\log |x|}= \dot c^t \quad (\text{for} \  n=2).
\end{cases}
\end{equation}
In addition, we have $$\dot u^t = \dot h^t + V^t\chi_{\Omega^t} \quad \text{in all}\   \R^n.$$
If moreover $k\ge 2$ then $\Psi$ is twice differentiable and we have
\begin{equation}  \label{psiddot}
\| \ddot \Psi^t \|_{C^{k-2,\alpha}(\R^n) } \le C_\circ
\end{equation}
and
\begin{equation}\label{estwt_restated}
\| \ddot u^t\|_{L^\infty(\R^n)} +\| \nabla \ddot u^t\|_{C^{k-2,\alpha}\left((\Omega^0\cup\Omega^t)^c\,\cup \,\overline{\Omega^0\cap\Omega^t}\right)} \le C_\circ.
\end{equation}

The constants $t_\circ$ and $C_\circ$ depend only on\footnote{The set of constants of the problem $\boldsymbol{\mathcal C}$ was defined in \eqref{constants}.} $\boldsymbol{\mathcal C}$.
\end{theorem}

A informal rephrasing of Theorem \ref{thm1} is as follows. 
If the moving obstacle $h(t,x)$ is $C^{k+1,\alpha}$ and $c(t)$ is $C^2$, then $\Gamma^t$ is``twice differentiable'' in  for $t$ in a small neighborhood of $0$. Moreover, the ``normal velocity'' of $\Gamma^t$ and the ``normal acceleration'' of $\Gamma^t$ are respectively $C^{k-1,\alpha}$ and $C^{k-2,\alpha}$ scalar fields on $\Gamma^t$, with the normal velocity precisely identified via a Dirichlet-to-Neumann transformation:  to compute it, one finds the solution $V^t$ to the Dirichlet problem in a exterior domain \eqref{diriV} and  the normal velocity at a point of $\Gamma^t$ is given by the normal derivative of $V^t$ divided by the Laplacian of the obstacle at that point.

\subsection{Open questions}
It is of course natural to ask whether similar results hold for more general obstacle problems, such as those associated to fully nonlinear operators or to fractional Laplacians.

In the view of our results\footnote{We establish that if $h\in C^{k+1,\alpha}$ then  $\Psi^t \in C^{k,\alpha}$ , $\dot\Psi^t \in C^{k,\alpha}$  and $\ddot\Psi^t \in C^{k-2,\alpha}$}
A natural open question, which we believe to be delicate, is whether one can improve Theorem \ref{thm1} to 
\[ \Psi(t,x) \in C^{k,\alpha} \quad \mbox{(jointly in $t$ and $x$)}.\]

\subsection{Structure of the proof and organisation of the paper}
For the proof, we first reduce to a situation where the contact set is growing, i.e. $\Omega^t\subset \Omega^0$. We then define a  coordinate system near the free boundary $\Gamma^0$, and express the ``height" $\eta^t$ of $\Gamma^t$ in these coordinates.
 
 In Section \ref{sec:apriori}, assuming that an expansion of the type $\eta^t= \eta^0 + \dot {\eta}^0 t+ \frac{1}{2} \ddot{ \eta}^0 t^2+ \dots$ holds as $t \to 0$, we derive equations for $\dot{ \eta}^0$ and $\ddot {\eta}^0$, which allow to obtain explicit formulae and  H\"older regularity for these quantities via single and double layers potential theoretic estimates. These regularity estimates are delicate to obtain because the relations characterizing $\dot {\eta}^0$ and  $\ddot{\eta}^0$ are  at first implicit and one needs to show they can be  ``closed" for regularity.
 
  In Section \ref{sec4}, we show that  the existence of an expansion in $t$  for $\eta^t$, which was previously assumed, does hold. This is done by using a second set of adapted coordinates near $\Gamma^0$ (a sort of hodograph transform) and again single and double layer potential estimates. 
  
  Finally, in Section \ref{sec5} we prove the main result by showing how to treat the general case where the contact set is not necessarily growing. In Appendix \ref{app}, we collect the potential theoretic estimates we need and  some additional proofs.

\vskip .5cm
{\bf Acknowlegments:} S. S. was supported by the Institut Universitaire de France and  NSF grant DMS-1700278. J.S. is supported by ERC Grant ``Regularity and Stability in Partial Differential Equations
(RSPDE)".

\section{Preliminaries}

\subsection{Known results}
Throughout the paper it is useful to quantify the smoothness of the (boundaries of the) domains $\Omega^t$. 
 Let us introduce some more notation with that aim. 
Let $U$ be some open set and $r>0$. 
We write $\partial U\in C^{k,\alpha}_r$ if for all $x_o\in \partial U$ there are some orthonormal coordinates $y_i$, $1\le i\le n$ with origin at $x_o$ (these coordinates may vary from point to point), and a function $F_{x_o}\in C^{k,\alpha}(\overline{B_{r}'})$ such that 
\[
 U\cap \big\{ |y'|<r, |y_n|<r \big\} = \{y_n< F_{x_o}(y')\} \cap\big\{|y'|<r, |y_n|<r \big\},
 \]
 where  $y'=(y_1, y_2, \dots, y_{n-1})$.
 
In this framework we denote
\begin{equation}\label{defCkalpharho}
\|\partial U\|_{C^{k,\alpha}_r} := \sup_{x_o\in \partial U} \|F_{x_o}\|_{C^{k,\alpha}(\overline{B_{r}'})} <\infty,
\end{equation}
where $B'_r =   \{|y'|<r\} \subset \R^{n-1}$.

With the previous assumptions we have in our notation
\begin{proposition}[\cite{Caf1, Caf2,  KN, Blank}]\label{blank}
There exist  universal constants $t_\circ>0$ and $C_o$ depending only on $\boldsymbol{\mathcal C}$ such that the following  hold.

(i) We have 
\[ \|\Gamma^t\|_{C^{k,\alpha}_{\rrho/4}} \le C_o\quad \mbox{for all }t\in (-t_\circ,t_\circ).\]

(ii)  For every pair $t, s\in (-t_\circ,t_\circ)$, the Hausdorff distance between $\Gamma^{t}$ and $\Gamma^s$ satisfies 
\[ d_{\rm Hausdorff}(\Gamma^{t},\Gamma^s) \le C_o\,|t-s|.\]
\end{proposition}
Proposition \ref{blank}  is contained in the results of \cite{Blank}. However, for the sake of completeness,   we briefly sketch  the proof in the appendix. This is done by combining the classical results for the obstacle problem in \cite{Caf1, Caf2,  KN} and the key sharp estimate $|\Omega^{t}\triangle  \Omega^s|\le C|t-s|$   for the symmetric difference of the positivity sets (or of the contact sets) from \cite{Blank}.

\subsection{Scalar parametrization of deformations (definition of $\eta^t$)}
By Proposition \ref{blank} the free boundaries $\Gamma^t $ are ``uniformly'' $C^{k, \alpha}$ for $|t|$ small and the difference between $\Gamma^{t} $ and $\Gamma^{s} $ is bounded by $C|t'-t|$ in $L^\infty$ norm. A goal of the paper is to prove that the difference is bounded  $C|t'-t|$ also in a  $C^{k-1, \alpha}$ norm. 
To prove this type of result it is convenient to  have a scalar function representing  the ``difference'' between $\Gamma^{t} $ and $\Gamma^{s} $. 
This has a clear meaning locally --- since both $\Gamma^{t}$ and $\Gamma^{s}$ are graphs,  and one can simply subtract the two functions that define these graphs. We  next give a global analogue of this. 

In a open neighborhood  $U_\circ$ of $\Gamma^0$ we define coordinates 
\[ (z,\s): U_\circ \longrightarrow \mathcal Z \times(-\s_\circ, \s_\circ),\]
where $\s_\circ>0$ and $\mathcal Z$ is some smooth approximation of $\Gamma^0$. 

We assume that the vector field
\[ N:= \partial_\s\]
is a smooth approximation of $\nu^0$ on $\Gamma^0$.  More precisely, we assume that 
\begin{equation}\label{defdelta0}
N\in C^\infty(U_\circ; \R^n), \quad  |N|= 1 \quad \mbox{and} \quad N\cdot \nu^t \ge (1-\varepsilon_o) \quad \mbox{for }t\in(-t_\circ,t_\circ),
 \end{equation}
where $\varepsilon_o$ is a constant that in the sequel will be chosen to be small enough ---depending only on  $\boldsymbol{\mathcal C}$.

In this framework, Proposition \ref{blank} implies that for all $t\in (-t_\circ,t_\circ)$ with $t_\circ$ small enough there exists $\eta^t\in C^{k,\alpha}(\mathcal Z)$ such that 
\begin{equation}\label{defeta}
 \Gamma^t \ = \ \{\s = \eta^t(z)\} \subset U_\circ.
 \end{equation}

\begin{remark}
From the data of $\Gamma^0$ we may always construct   $\mathcal Z$ and $(z,\s)$ satisfying the previous properties --- for $\varepsilon_o$ arbitrarily small --- by taking $\mathcal Z$ to be a smooth approximation of $\Gamma^0$ and $N$ a smooth approximation of $\nu^0$.
Once $\mathcal Z$ and $N$ are chosen, the coordinates $(z,\s)$ are then defined respectively as the projection on $\mathcal Z$ and the signed distance to $\mathcal Z$ along integral curves of $N$. 
\end{remark}

\section{A priori estimates}\label{sec:apriori}

Roughly speaking, the goal of this section is to show that if an expansion of the type 
\[\eta^t = \eta^0 + \dot\eta^0 t+ \frac 1 2 \ddot \eta^0 t^2 + \ \cdots,\]
holds, where 
\[ \frac{\eta^t-\eta^0}{t} \rightarrow \dot \eta^0 \quad \mbox{and}\quad \frac{\eta^t -\eta^0- \dot \eta^0 t}{t^2} \rightarrow \frac{\ddot \eta^0}{2} \quad\mbox{ as $t\to 0$, in }C^{0}(\mathcal Z)\] 
then $\dot \eta^0$ and $\ddot \eta^0$ must satisfy certain equations that have uniqueness of solution and a priori estimates.
From these equations we obtain conditional  (or a priori) estimates for $\|\dot \eta^0\|_{C^{k-1,\alpha}(\mathcal Z)}$ and $\|\ddot \eta^0\|_{C^{k-2,\alpha}(\mathcal Z)}$.

In the next sections, let us provisionally assume that 
\begin{equation}\label{orderingassump}
\Delta (h^t -h^0)\ge 0\quad \mbox{and}\quad   c^t-c^0\le 0 
\end{equation}
for all $t\ge 0$, which is not essential but simplifies the analysis: 
Assumption  \eqref{orderingassump} guarantees that  $\Omega^{t}\subset\Omega^0$ for all $t\ge0$.
Indeed, this is an immediate consequence of the characterisation of 
\[ \tilde u^t := u^t-h^t\]
 as the infimum of all nonnegative supersolutions with the same right hand side and appropriate condition at infinity. 
More precisely, we have the following lemma, whose  proof  is standard in dimension $n\ge 3$ and which we sketch in dimension $n=2$ in the appendix. 

\begin{lemma}\label{lemcharinf}
The function $\tilde u^t$ can be defined as the infimum of all $f$ satisfying 
\begin{equation}\label{tildeu}
f \ge 0, \quad \Delta f \le -\Delta h^t, \quad \mbox{and}\quad  \lim_{x\to \infty} (f +h^t) \ge  c^t, \quad \text{resp.}  \ \lim_{x\to \infty} \frac{f+h^t}{-\log |x|} \ge c^t .
\end{equation}
\end{lemma}

Note that in particular $f= \tilde u^0$ satisfies \eqref{tildeu} since $\Delta u^0 = -\Delta h^0 \le -\Delta h^t$, and 
$$\lim_{x\to \infty} (\tilde u^0 +h^t) =  \lim_{x\to \infty} (\tilde u^0 +h^0)+\lim_{x\to \infty} (h^t-h^0)\ge  c^0 \ge c^t,$$ resp. 
$$\lim_{x\to \infty}\frac{ (\tilde u^0 +h^t) }{-\log |x|} \ge c(t).$$ 
 Therefore, applying Lemma \ref{lemcharinf} we obtain
$\tilde u^0 \ge \tilde u^t$ and  
\[\Omega^t = \{\tilde u^t >0\} \subset \{\tilde u^0 >0\} = \Omega^0\] 
for all $t>0$.
Equivalently  \eqref{orderingassump} implies that $\eta^t\ge 0$ on $\mathcal Z$ for $t>0$.

Later, when we prove Theorem \ref{thm1}, we will reduce to this case by decomposing $h^t$ as a sum of two functions, one  with nonnegative Laplacian and one with nonpositive Laplacian.

Let us define
\begin{equation}\label{tildev}
v^t := \delta_t \tilde u^0 = \frac 1 t  (\tilde u^t - \tilde u^0).
\end{equation}
The function $v^t$ is a solution of
\begin{equation}\label{vt}
\begin{cases}
 \Delta v^t  = -\Delta \delta_t h^0 \quad &\mbox{in }\Omega^t
\\
v^t =  -\frac 1 t  \tilde u^0   \quad &\mbox{on }\Gamma^t
\\
\lim_{x\to \infty} v^t = \delta_t c^0,\quad \text{resp. }  \lim_{x\to \infty} \frac{v^t(x)}{-\log |x|}= \delta_t c^0.
\end{cases}
\end{equation}

Since $\tilde u^0= | \nabla \tilde u^0| = 0$ on $\Gamma^0$, using  the classical estimate\footnote{Since $u^0$ is a  solution of the obstacle problem in  the whole $\R^n$ with a semiconcave obstacle $h^0$, $u^0$ is semiconcave with $D^2 u^0\ge -\|h\|_{C^{1,1}(\R^n)} {\rm Id}$ and the estimate follows using $\Delta u^0 =0$ where $u^0>h^0$.}
\[\|u^0\|_{C^{1,1}(\R^n)} \le (n-1) \|h^0\|_{C^{1,1}(\R^n)},
\] 
we obtain
\[
|\tilde u^0|\le C  \|h\|_{C^{1,1}(\R^n)} \,d_{\rm Hausdorff}^2(\Gamma^t,\Gamma^0) \le C t^2\quad\mbox{on }\Gamma^t.
\] 
Then, using that $\Omega^t$ grows to $\Omega^0$ as $t\downarrow 0$ and uniform estimates for $v^t$ we find that  $v^t\rightarrow v$ as $t\downarrow 0$, where $v$ is the solution of
\begin{equation}\label{eqnv}
\begin{cases}
\Delta  v  = -\Delta\dot h^0 \quad &\mbox{in }\Omega^0
\\
v =  0 &\mbox{on }\Gamma^0
\\
v (\infty)=  \dot c^0,   & \quad \text{resp. }  \lim_{x\to \infty} \frac{v(x)}{-\log |x|}= \dot c^0.
\end{cases}
\end{equation}
Here $ \Delta \dot h^0= \lim_{t\downarrow 0} \Delta \delta_t h^0  = (\Delta \partial_t  h)(0,x)$.  


\subsection{Equation and estimate for $\dot \eta^0$}

We first prove the following
\begin{proposition}\label{apriori1}
Let $k\ge 1$. Assume that for some  $t_m\downarrow 0$ there exists $\dot\eta^0\in C^{0}(\mathcal Z)$  such that 
\[\delta_{t_m}\eta^0\rightarrow \dot\eta^0\quad \mbox{in }C^{0}(\mathcal Z) \qquad \mbox{as } m \to \infty. \]

Then, the limit $\dot\eta^0$  is given by 
\begin{equation}\label{eqndoteta0}
 \dot \eta^0(z)  =  \left( \frac{\partial_N v}{ (N\cdot \nu^0)^2\Delta h^0 }\right) (z,\eta^0(z)) ,
\end{equation}with $v$ as in \eqref{eqnv}.
As a consequence, $\dot \eta^0$  is independent of the sequence $t_m$  and we have $\dot\eta\in C^{k-1,\alpha}(\mathcal Z)$  with the estimate
\begin{equation}\label{estdoteta0}
 \|\dot \eta^0\|_{C^{k-1,\alpha}(\mathcal Z)} \le C(\boldsymbol{\mathcal C^0}) \big(\|\dot h^0\|_{C^{k,\alpha}(B_{\RR})} + |\dot c^0|\big).
 \end{equation}
\end{proposition}
\begin{proof} We split the proof in two steps. 

{\em Step 1.} We prove  \eqref{eqndoteta0}.
Recall that since $\tilde u^t$ is a solution of a zero obstacle problem we have 
\[\tilde u^t =| \nabla \tilde u^t|= 0\quad \mbox{on }\Gamma^t.\] 
Thus,
\begin{equation}\label{crucialrelation1}
 \partial_{\s} v^t  =  \frac 1 t \big( \partial_{\s} \tilde u^t - \partial_{\s} \tilde u^0 \big)  = -\frac{ \partial_{\s} \tilde u^0}{t} \quad \mbox{on }\Gamma^t.
 \end{equation}

From \eqref{crucialrelation1} we deduce that 
 \begin{equation}\label{crucialrelation2}
 \partial_{\s} v^{t_m}(z,\eta^{t_m})  = - \frac{1}{t_m}\partial_{\s} \tilde u^0(z,\eta^{t_m}) = - \frac{1}{t_m} \big( \partial_{\s} \tilde u^0(z,\eta^0) +   \partial_{\s\s} \tilde u^0(z,\eta^0) (\eta^{t_m}-\eta^0) + o(t_m)\big)
 \end{equation}
 where $\eta^0$ and $\eta^{t_m}$ are evaluated at $z$ (although we omit this in the notation) and 
 where  $\partial_{\s\s} \tilde u^0(z,\eta^0)$ is understood as the limit from the $\Omega^0$ side. 
 To justify the validity of the previous Taylor expansion we  use that $\tilde u^0 \in C^{2,\alpha}(\overline{\Omega^0})$, see Lemma \ref{lemhigherregtildeu}.
 
Since $\tilde u^0= |\tilde \nabla u^0| = 0$ on $\Gamma^0$ we obtain 
\[
\partial_{ee}  \tilde u^0 = (e\cdot \nu)^2 \partial_{\nu\nu} \tilde u^0 = (e\cdot \nu)^2 \Delta \tilde u^0 = -(e\cdot \nu)^2\Delta h^0 \quad \mbox{on }\Gamma^0 
\]
for every vector $e$, where  $\nu=\nu^0$ is the normal vector to $\Gamma^0$ (pointing towards $\Omega^0$). 
Again, the previous second derivatives on $\Gamma^0$ mean the limits from the $\Omega^0$ side.
Hence, we have
 \begin{equation}\label{crucialrelation3}
 \partial_{\s} \tilde u^0(z,\eta^0(z))=0 \qquad \mbox{and }\qquad   \partial_{\s\s} \tilde  u^0(z,\eta^0(z))  =  -\big((N\cdot \nu^0)^2\Delta h^0 \big)(z,\eta^0(z)) 
 \end{equation}
where $  \partial_{\s\s} \tilde u^0 (z,\eta^0(z))$ is from the $\Omega^0$ side. 
Dividing \eqref{crucialrelation3} by $t_m$ and taking the limit as $t_m\downarrow 0$ in \eqref{crucialrelation2} using the assumption,  we obtain
 \begin{equation}\label{crucialrelation4}
 \partial_{\s} v(z,\eta^0(z))  =   -\partial_{\s\s} \tilde u^0(z,\eta^0(z)) \,\dot\eta^0(z) =  \big((N\cdot \nu^0)^2\Delta h^0 \big)(z,\eta^0(z))  \,\dot\eta^0(z)
 \end{equation}
 where $\partial_{\s} v(z,\eta^0(z))$ and  $\partial_{\s\s} \tilde u^0(z,\eta^0(z))$ are  from the $\Omega^0$ side.
When computing the limit that yields \eqref{crucialrelation4} we must check that 
\begin{equation}\label{star1}
 \partial_{\s} v^{t_m}(z,\eta^{t_m}(z)) \rightarrow  \partial_{\s} v(z,\eta^0(z)),
 \end{equation}
  where $\partial_{\s} v(z,\eta^0)$ is from the $\Omega^0$ side. To prove this, note that the equation \eqref{vt} for $v^t$ with the uniform $C^{1,\alpha}$ estimates for the boundaries $\Gamma^t$ imply that $\|\nabla v^t\|_{C^{0,\alpha}(\overline{\Omega^t})}$  is uniformly bounded (for $t>0$ small). This implies that $\nabla v^t$ converges uniformly to $\nabla v$ in every compact set of $\Omega^0$. Then using the uniform continuity of the derivatives of $v$ on $\overline{\Omega^0}$ we show that 
\[
\lim \nabla v^{t_p} (x_p) \rightarrow \nabla v(x) \quad \mbox{as }p\to \infty \quad \mbox{whenever }t_p\downarrow 0 , \ x_p\to x\ \mbox{ and }x_p\in \overline{\Omega^{t_p}}.
\]
This establishes \eqref{star1} and  \eqref{crucialrelation4}. Then,  \eqref{eqndoteta0} follows immediately form \eqref{crucialrelation4}, after recalling that $N= \partial_{\s}$.

{\em Step 2.} We prove \eqref{estdoteta0}.
Indeed, from \eqref{eqnv}, and using that $\Gamma^0 =\partial \Omega^0\in C^{k,\alpha}_{\rrho/4}$ with norm universally bounded we obtain that 
\begin{equation}\label{estv}
\|v\|_{C^{k,\alpha}(\Omega^0)}\le C(\boldsymbol{\mathcal C^0}) \big(\|\Delta \dot h^0\|_{C^{k-2,\alpha}(\Omega^0)} + |\dot c^0|\big)
\le  C(\boldsymbol{\mathcal C^0}) \big(\| \dot h^0\|_{C^{k,\alpha}(\Omega^0)} + |\dot c^0|\big).
\end{equation}

Now recalling that $N$ is smooth, that $\|\nu^0\|_{C^{k-1,\alpha}(\Gamma^0)} \le  C \|\Gamma^0\|_{C^{k,\alpha}_{\rrho/4}} \le C$, that 
$-\Delta h^0\ge \rrho$, and that $\|\eta^0\|_{C^{k,\alpha}(\mathcal Z)}\le C$,  \eqref{eqndoteta0}  and \eqref{estv} imply  \eqref{estdoteta0}.
\end{proof}

\subsection{Equation and estimate for $\ddot \eta^0$}

In this section we estimate the second derivative in $t$ of $\eta$ at $t=0$. It is convenient to introduce here the following notation, that we shall use throughout the paper. Given a function  $f:(-t_\circ, t_\circ)\times Y\to \R$  recall the notation  $f= f^t(y) = f(t,x)$. 
Let us also denote
\[
\delta_t^2 f^s :=  2\frac{\delta_tf^s -\dot f^s}{t}   \qquad \mbox{and} \qquad \ddot f^s :=  \lim_{t\downarrow 0}\delta_t^2 f^s =  \partial_{tt} f(y,0).
\]

From now on let us consider $v$ to be defined in all of $\R^n$ by extending the solution of \eqref{eqnv} by $0$ in $\R^n\setminus \Omega^0$. Note that this is consistent with $v= \lim_{t\downarrow0} v^t $ and  $v^t= \delta_t \tilde u^0 =0$ in $\R^n\setminus \Omega^0$ (since both $\tilde u^t$ and $\tilde u^0$ vanish there).

We now introduce the function, defined in all of $\R^n$, 
\[
w^t := \delta_t v^0 = \frac 1 t \big( v^t-v \big) = \frac 1 2 \delta_t^2 \tilde u ^0.
\]

Using \eqref{eqnv} and the following identity 
\[\Delta v^t= \frac{1}{t} \Delta (\tilde u^t-\tilde u^0)=- \frac{1}{t}\Delta \tilde u^0 = \frac{1}{t}\Delta h^0 \quad \mbox{in } \Omega^0\setminus \Omega^t\]
we find, in the distributional sense,
\begin{equation}\label{eqnwt}
\begin{cases}
\Delta w^t  =  \frac 1 t \left(\frac{\partial_Nv}{N\cdot \nu^0} \, \mathcal{H}^{n-1}\restriction_{\Gamma^0} + \left(\frac 1 t \Delta h^0 - \Delta \dot h^0\right)\, \chi_{\Omega^0\setminus \Omega^t} \right) - \frac 1 2 \Delta \delta_t^2 h^0 \chi_{\Omega^t}\quad & \mbox{in }\R^n
\vspace{5pt}
\\
w^t (\infty)=  \frac 1 2  \delta^2_t c^0,  \quad \text{resp.} \ \lim_{x\to \infty} \frac{w^t}{-\log |x|}=  \frac 1 2  \delta^2_t c^0
\end{cases}
\end{equation}
where $\mathcal{H}$ denotes the Hausdorff measure. 
Indeed,  note also that for $\nu = \nu^0$ we have
\[
 \partial_{N} v = (N\cdot \nu^0) \partial_\nu v\quad\mbox{on }\Gamma^0_{\rm out}\qquad \mbox{while}\qquad \partial_{\nu} v =0\quad  \mbox{on } \Gamma^0_{\rm in}.
\]  
Here,  `` $\Gamma^0_{\rm out}$'' refers to the limit from the $\Omega^0$ side while `` $\Gamma^0_{\rm in}$'' refers to the limit from the $\R^n\setminus\overline{\Omega^0}$ side.
Therefore, $\Delta w^t$ has  some mass concentrated on $\Gamma^0$ which is given by the jump in the normal derivative of $v$, namely, 
\[ \frac{1}{t} \frac{\partial_Nv}{N\cdot \nu^0} \, \mathcal{H}^{n-1}\restriction_{\Gamma^0}.\]

In the following lemma, and throughout the paper, $P$ denotes the Newtonian potential in dimension $n$,  namely:
\[
P(x) = \frac{1}{n(n-2)|B_1|}  |x|^{2-n}\qquad \mbox{resp. } P(x) = -\frac{1}{2\pi} \log |x|.
\]
Recall that $-\Delta  P = \delta_{x=0}$ in the sense of distributions.

We also need to introduce the Jacobian 
\[J(z,\s) := | {\rm det}\,  D\, (z, \s)^{-1}|\]
of the coordinates $(z,\s)$ defined by
\[ \int_A f(x) \,dx = \int_{(z,\s)(A)} f(z,\s) J(z,\s) \, dz\,d\s.\]
We use the following abuse of notation:
\begin{itemize} 
\item when $f=f(x)$ we  denote $f(z,\s)$ the composition $f\circ(z,\s)^{-1}$ ; and conversely,  
\item when $g=g(z,\s)$ we will denote $g(x)$ the composition $g\circ(z,\s)$. 
\end{itemize}

Finally, let us denote
\[ \pi_1 : U_\circ \rightarrow \mathcal Z\]
the projection map along $N$, which is defined in the coordinates $(z,\s)$ by 
\[ (z,\s)\mapsto (z,0).\]

We will need the following 
\begin{lemma} \label{LemmaJacobian}
Given $f:\Gamma^0\rightarrow \R$ continuous we have 
\[
\int_{\Gamma^0} (N\cdot \nu^0)(x) f(x) d\mathcal{H}^{n-1}(x)  = \int_{\mathcal Z} f(z,\eta^0(z)) J(z,\eta^0(z)) dz. 
\]
\end{lemma}
\begin{proof}
Let us assume without loss of generality that $f$ is defined and continuous in the neighborhood $U_\circ$ of $\Gamma^0$.
Given $\varepsilon>0$ let $$A^\varepsilon: = \{x\in U_\circ\ :\  \eta^0(z(x)) \le s(x) \le \eta^0(z(x))+\varepsilon\}.$$
Recalling that $N= \partial_s$ and that $|N| =1$,  we have
\[
\int_{\Gamma^0} (N\cdot \nu^0)(x) f(x) d\mathcal{H}^{n-1}(x)  =\lim_{\varepsilon \downarrow 0}  \frac 1\varepsilon  \int_{A^{\varepsilon}} f(x) d\mathcal{H}^n(x).
\]

On the other hand, for $(z,s)(A^\varepsilon) : = \{(z,\s)\in\mathcal Z\times (-s_\circ,s_\circ):\  \eta^0(z) \le \s \le \eta^0(z)+\varepsilon\}$ we have, by definition of $J$,
\[
\begin{split}
\frac 1\varepsilon \int_{A^{\varepsilon}} f(x) d\mathcal{H}^n(x)& = \frac 1\varepsilon \int_{(z,s)(A^\varepsilon)} f(z,\s) J(z,\s) dz\,d\s 
\\
&=  \int_{\mathcal Z} dz\, \frac 1\varepsilon \int_{0}^{\varepsilon} d\bar s \,f(z,\eta^0(z)+\bar s) J(z,\eta^0(z)+\bar s)
\\
&= \int_{\mathcal Z} f(z,\eta^0(z)) J(z,\eta^0(z))\, dz + o(1)
\end{split}
\]
as $\varepsilon\downarrow 0$ and the lemma follows.
\end{proof}

  
\begin{lemma} \label{distlimit}
Let $k\ge 2$.  Assume that for some $t_m\downarrow 0$ there exist $\dot\eta, \ddot\eta\in C^{0}(\mathcal Z)$ such that  
\[ \delta^2_{t_m} \eta^0 = 2\frac{\eta^{t_m} -\eta^0- \dot \eta^0 t_m}{t_m^2} \rightarrow \ddot \eta^0 \quad\mbox{in }C^{0}(\mathcal Z)\] 
as $t_m\downarrow 0$. 
Then, 
\[
 w^{t_m} \quad   \xrightarrow{\rm \ weakly\ }  \quad  w \quad \mbox{in } \R^n
\]
where $w$ can be decomposed as 
\begin{equation}\label{decompw}
w = w_{\rm solid} + w_{\rm single} + w_{\rm double}+ w_{\rm implicit} +  \mbox{constant}
\end{equation} 
for
\begin{equation}\label{solid}
 w_{\rm sol.}(x) =   \int_{\R^n } d\mathcal H^n(y) \big(\Delta \ddot h^0\, \chi_{\Omega_0}  \big)(y) P(x-y),
\end{equation}
\begin{multline}\label{single}
w_{\rm sin.}(x)
 :=   \int_{\Gamma_0} d\mathcal{H}^{n-1}(y)\, \bigg(  (N\cdot\nu^0) (\dot \eta^0\circ\pi_1)\,\Delta\dot h^0\, \\- \frac 1 2 (\dot \eta^0\circ \pi_1)^2  \frac{N\cdot\nu^0}{J} \partial_N\, \big(J \Delta h^0)\bigg)(y) \,P(x-y),
\end{multline}
\begin{equation}\label{double}
w_{\rm dou.}(x) :=   \int_{\Gamma_0} d\mathcal{H}^{n-1}(y)  \frac 1 2 \left( (\dot \eta^0\circ \pi_1)^2  \frac{N\cdot\nu^0}{J} \, \big(J \Delta h^0)\right)(y) \, \partial_N P(x-y),
\end{equation}
\begin{equation}\label{implicit}
w_{\rm imp.}(x) :=   \int_{\Gamma_0} d\mathcal{H}^{n-1}(y) \frac{\Theta}{(N\cdot \nu) }  (y)\, P(x-y),
\end{equation}
where $\Theta : \Gamma^0 \rightarrow \R$
\begin{equation}\label{defTheta}
\Theta  :=  \frac 1 2 (N\cdot \nu^0)^2 \Delta h^0 \, (\ddot \eta^0 \circ \pi_1).
\end{equation}
\end{lemma}
\begin{proof}
Define
\[
\mathcal D^{t} :=  \Delta w^{t} = \frac 1 t \frac{\partial_N\,v}{N\cdot \nu^0} \, \mathcal{H}^{n-1}\restriction_{\Gamma^0}  - \left(\frac{1} {t^2} \Delta h^0 + \frac{1}{t}\Delta \dot h^0\right) \chi_{\Omega^0 \setminus \Omega^t} - \frac 1 2 \Delta \delta_t^2 h^0 \chi_{\Omega^t}.
\]

Let us show that $\mathcal D^{t_m} \rightarrow \mathcal D$  in the sense of distributions, for some distribution $\mathcal D$ that we compute.

Let us first  write 
\[ 
\mathcal D^t = \mathcal D^t_1 +\mathcal D^t_2
\]
where 
\begin{equation}\label{Dt1}
\mathcal D^t_1  :=    -\frac 1  t  \Delta \dot h^0 \, \chi_{\Omega^0\setminus \Omega^t}  - \frac 1 2 \Delta \delta_t^2 h^0 \chi_{\Omega^t}
\end{equation}
and 
\begin{equation}\label{Dt2}
\mathcal D^t_2  :=    \frac 1 t \bigg(\frac{\partial_N v}{ N\cdot \nu^0}\, \mathcal{H}^{n-1}\restriction_{\Gamma^0} - \frac 1 t \Delta h^0 \, \chi_{\Omega^0\setminus \Omega^t} \bigg).
\end{equation}
First we clearly have, for $\phi\in C^\infty_c(\R^n)$,
\[
\begin{split}
\int \phi(x) &\left(\frac 1  t  \Delta \dot h^0 \, \chi_{\Omega^0\setminus \Omega^t} \right)(x) \, dx =
\frac{1}{t} \int_{\mathcal Z}\int_{\eta^0}^{\eta^t}J(z,\s)  (\Delta\dot h^0\,\phi)(z,\s)\, dz \, d\s   
\\
 &\longrightarrow   \int_{\mathcal Z}\dot\eta^0(z) J(z,\eta^0)  (\Delta\dot h^0\,\phi)(z,\eta^0)\, dz = {\int_{\Gamma^0}} (N\cdot \nu^0) (\dot \eta^0\circ\pi_1)\,\Delta\dot h^0\,\phi \, d\mathcal{H}^{n-1}
\end{split}
\]
as $t=t_m \downarrow 0$, where we have used Lemma \ref{LemmaJacobian} and hence
\begin{equation}\label{AA}
\mathcal D^{t_m}_1  \ \xrightarrow{weakly}  - (N\cdot \nu^0) (\dot \eta^0\circ\pi_1) \,\Delta\dot h^0\,\mathcal{H}^{n-1}\restriction_{\Gamma^0} - \frac 1 2 \Delta \ddot h^0 \chi_{\Omega^0}.
\end{equation}

Next, using \eqref{eqndoteta0}, we compute, for $\phi\in C^\infty_c(\R^n)$,
\begin{equation}\label{BB}
\begin{split}
\int \phi \mathcal D^t_2 &=
\frac 1 t \left(\ \int_{\mathcal Z}dz  \frac{J(z,\eta^0)}{(N\cdot \nu^0)^2 (z,\eta^0)}  \partial_N v(z) \phi(z,\eta^0)  - \ \int_{\mathcal Z}dz    \frac 1 t \int_{\eta^0}^{\eta^t} d\s( J \Delta h^0\, \phi)(z,\s) \right) 
\\&=
 \frac 1 t \int_{\mathcal Z}dz  \left( \big(J \Delta h^0\, \phi   \big) (z,\eta^0)\dot \eta^0    -\frac 1 t \int_{\eta^0}^{\eta^t} d\s( J \Delta h^0\, \phi)(z,\s)\right) 
 \\&=
 I_1 + I_2, 
\end{split}
\end{equation}
where
\[I_1 :=   \frac 1 t \ \int_{\mathcal Z}dz  \left( \big(J \Delta h^0\, \phi   \big) (z,\eta^0)\dot \eta^0-    \frac 1 t \int_{\eta^0}^{\eta^0+\dot\eta^0 t} d\s( J \Delta h^0\, \phi)(z,\s)\right)  \]
and 
\[I_2 := \frac{-1}{t^2}   \int_{\mathcal Z}dz    \int_{\eta^0+\dot\eta^0 t}^{\eta^t} d\s( J \Delta h^0\, \phi)(z,\s). \]

On the one hand, letting $s = \eta^0 + \dot\eta^0 t \bar s $, 
\begin{equation}\label{CC}
\begin{split}
I_1 &=  \int_{\mathcal Z}dz  \int_0^1 \dot \eta^0 (z) t \,d \bar \s    \, \frac{\bar \s}{t}\left(\frac{ \big(J \Delta h^0\, \phi   \big) (z,\eta^0) -   ( J \Delta h^0\, \phi)(z,\eta^0+\dot\eta^0 t \bar\s )}{\bar \s t}\right) 
\\&=  \int_{\mathcal Z}dz  \int_0^1 (\dot \eta^0)^2 (z) \bar \s  d\bar \s\,   \partial_\s  \big(J \Delta h^0\, \phi   \big) (z,\eta^0) + o(1)
\\&=  \int_0^1  \bar \s  d\bar \s  \int_{\mathcal Z}dz  (\dot \eta^0)^2 (z)  \partial_\s  \big(J \Delta h^0\, \phi   \big) (z,\eta^0) + o(1)
\\&=   \frac 1 2 \int_{\Gamma^0}  (N\cdot\nu^0) \, d\mathcal{H}^{n-1}   \,\frac 1 J(\dot \eta^0\circ \pi_1)^2 \,  \partial_N\, \big(J \Delta h^0\, \phi   \big) +o(1).
\end{split}
\end{equation}
as $t=t_m \downarrow 0$, where for the last relation we used Lemma \ref{LemmaJacobian} with 
\[f(x) =  \left( \frac 1 J(\dot \eta^0\circ \pi_1)^2 \partial_N\, \big(J \Delta h^0\, \phi   \big) \right) (x) ,\]  noting also   that $\partial_s = \partial_N$ and $(\dot \eta^0\circ \pi_1)^2(z,\eta^0(z))  =   (\dot \eta^0)^2 (z)$ .
On the other hand, 
\begin{equation}\label{DD}
\begin{split}
I_2 &=  \frac{-1}{t^2}   \int_{\mathcal Z}dz    \int_{\eta^0+\dot\eta^0 t}^{\eta^0+\dot\eta^0 t +\frac 1 2 \ddot\eta^0t^2} d\s( J \Delta h^0\, \phi)(z,\s) +o(1)
\\&=
 -\frac 1 2 \int_{\mathcal Z}dz  \,\ddot\eta^0 \,( J \Delta h^0\, \phi)(z,\eta^0)  +o(1)
 \\
 &=
 -\frac 1 2 \int_{\Gamma^0} d\mathcal{H}^{n-1}  (N\cdot \nu^0)\,(\ddot\eta^0\circ\pi_1) \, \Delta h^0 \phi\, +o(1)
\end{split}
\end{equation}
as $t=t_m \downarrow 0$. 
Therefore, $\mathcal D^{t_m}_2 \rightarrow \mathcal D_2$ where 
\begin{equation}\label{D2}
\int \phi \mathcal D_2 =
\frac 1 2 \int_{\Gamma^0} d\mathcal{H}^{n-1}  (\dot \eta^0\circ \pi_1)^2 \, \frac{N\cdot\nu^0}{J} \partial_N\, \big(J \Delta h^0\, \phi   \big) 
-\frac 1 2 \int_{\Gamma^0} d\mathcal{H}^{n-1}  (N\cdot \nu^0)\,(\ddot\eta^0\circ\pi_1) \, \Delta h^0\phi.
\end{equation}

In dimension $n\ge 3$ we have
\[
w^{t_m}(\infty)= \frac 1 2\lim_{x\to \infty} \delta^2_{t_m} \tilde u ^0(\infty)    = \frac 1 2 \delta^2_{t_m} c ^0 \rightarrow \frac1 2 \ddot c^0 
\] 
and thus 
\[
w(\infty) = \frac 1 2 \ddot c^0 = \mbox{constant}.
\]
In dimension $n=2$ we have  instead
\[
\lim_{x\to \infty} \frac{w(x)}{- \log|x|}  = \frac1 2 \ddot c^0 
\]
and this implies that $2\pi \frac1 2 \ddot c^0 =  \int_{\R^2} \Delta w$ and that $w$ can be obtained (up to an additive constant) by convoling the Newtonian potential $P$ with $\Delta w$.

Therefore, combining \eqref{AA} and \eqref{D2},  we obtain that \eqref{decompw}--\eqref{implicit} hold.
\end{proof}

We may now state the  final result of this section.
\begin{proposition}\label{propapriori2}
Let $k\ge 2$.  Assume that for some $t_m\downarrow 0$ there exist $\dot\eta, \ddot\eta\in C^{0}(\mathcal Z)$ such that  
\[ \delta^2_{t_m} \eta^0 = 2\frac{\eta^{t_m} -\eta^0- \dot \eta^0 t_m}{t_m^2} \rightarrow \ddot \eta^0 \quad\mbox{in }C^{0}(\mathcal Z)\] 
as $t_m\downarrow 0$. 
Assume that $w\in C^1(\overline{\Omega^0})$ and
\begin{equation}\label{convergencewtk}
\lim \nabla w^{t_m} (x_m) \rightarrow \nabla w(x) \quad \mbox{as }m\to \infty \quad \mbox{for all } \ x_m\to x\ \mbox{ such that  }x_m\in \overline{\Omega^{t_m}}.
\end{equation}
Then,  $\Theta : \Gamma^0 \rightarrow \R$ defined by \eqref{defTheta} satisfies
\begin{equation}\label{second2}
\Theta -\frac 1 2 \partial_{\s\s\s} u^0 (\dot \eta^0 \circ\pi_1)^2  = \partial_{ss}v\, \dot \eta^0 + \partial_\s w \quad\mbox{on }\Gamma^0_{\rm out}.
\end{equation}
Moreover,  $\ddot \eta^0$ does not depend on $(t_m)$  and
\begin{equation}\label{estddoteta0}
 \|\ddot \eta^0\|_{C^{k-2,\alpha}(\mathcal Z)} \le C(\boldsymbol{\mathcal C}^0) \boldsymbol {\mathcal Q}
 \end{equation}
 where 
\[\boldsymbol {\mathcal Q}:=  \|\ddot h^0\|_{C^{k-1,\alpha}(\R^n)} + |\ddot c^0| +(\|\dot h^0\|_{C^{k,\alpha}(\R^n)}+|\dot c^0|)(  \|\dot h^0\|_{L^\infty(\R^n)}+|\dot c^0|).\]

\end{proposition}
As for $\dot\eta$, the independence of $t_m$ and regularity of $\ddot\eta$ will be consequences of the fact that $\Theta$ solves the equation \eqref{second2}, for which regularity estimates and uniqueness hold. However, note that \eqref{second2} is an implicit equation for $\Theta$ since $w_{\rm imp.}$ depends on $\Theta$, which makes the analysis more involved.

To prove Proposition \ref{propapriori2}, we will need two auxiliary lemmas with standard proofs.
\begin{lemma}
\label{lemhigherregtildeu}
We have 
\[
\|\tilde u^0\|_{C^{k+1,\alpha}\left(\overline{B_{\RR}\cap\Omega^t}\right)}\le C(\boldsymbol{\mathcal C}^0 )
\] 
More generally, for $t\in[0,t_\circ)$ , where $t_\circ=t_\circ(\boldsymbol{\mathcal C})$ we have $\tilde u^t \in C^{k+1,\alpha}\left(\overline{\Omega^t}\right)$ with 
\[
\|\tilde u^t\|_{C^{k+1,\alpha}\left(\overline{B_{\RR}\cap\Omega^t}\right)}\le C(\boldsymbol{\mathcal C} )
\] 
\end{lemma}
\begin{proof}
Note that $\partial_i \tilde u^t$ solves 
 \[ \Delta (\partial_i \tilde u^t) = -\Delta (\partial_i  h^t)\quad \mbox{in }\Omega^t\quad \mbox{ with } \quad \partial_i \tilde u^t= 0 \quad \mbox{on }\Gamma^t= \partial \Omega^t .\]
Since $-\Delta (\partial_i  h^t) \in  C^{k-2,\alpha}(\R^n)$ and  $\Gamma^t$ belongs to $C^{k,\alpha}_\rrho$, using standard Schauder estimates up to the boundary we obtain \[ \partial_i \tilde u^t \in C^{k,\alpha}\left(\overline{B_{\RR}\cap\Omega^t}\right) \]
 and hence 
\[
\tilde u^t \in C^{k+1,\alpha}\left(\overline{B_{\RR}\cap\Omega^t}\right).
\] 
\end{proof}

\begin{lemma}\label{potofchar}
Let $U\subset \overline{B_{R}}\subset\R^n$ be  bounded with $\partial  U$ belonging to  $C^{m+2,\alpha}_{r}$ for some $r>0$ and $f\in C^{m,\alpha}_c(B_{2R})$, where $m\ge 0$.
Let $W$  the solution of 
\[
\begin{cases}
\Delta W =f\,\chi_{\R^n\setminus U} \quad&\mbox{in } \R^n
\\
W(\infty) = 0 & \quad \text{resp. } \lim_{x\to \infty} \frac{W(x)}{-\log |x|} = 2\pi \int_{\R^2} f\,\chi_{\R^n\setminus U} \  .
\end{cases}
\]
which is given in dimension 2 by convolution with the logarithmic Newtonian potential.

Then, 
\[
\|W\|_{C^{m+2,\alpha}(\overline{B_{2R}} \setminus U)}\ +\|W\|_{C^{m+2,\alpha}(\overline{U})}\le C  \|f\|_{C^{m,\alpha}(\overline{B_{2R}})}
\]
where $C = C(n, m, \alpha,R,  r, \|\partial U\|_{C^{m+2,\alpha}_{r}})$.
\end{lemma}
\begin{proof}
 Let $\tilde W$ be the solution of 
\[
\begin{cases}
\Delta \tilde W =f \quad&\mbox{in } \R^n\setminus U
\\
\tilde W = 0 &\mbox{on }\partial U
\\
\tilde W(\infty) = 0, &\quad \mbox{resp. } \lim_{x\to \infty} \frac{\tilde W(x)}{-\log |x|} = 0 \ .
\end{cases}
\]
We consider $\tilde W$ defined in all of $\R^n$ by extending it by $0$ in $U$.

Note that by standard Schauder estimates up to the boundary we have
\begin{equation}\label{star2}
\| \tilde W \|_{C^{m+2,\alpha}(\overline{B_{2R}}\setminus U)} \le C \|f\|_{C^{m,\alpha}(\overline{B_{2R}})}.
\end{equation}

On the other hand the difference $(\tilde W-W)$ solves, in all of $\R^n$
\[
\begin{cases}
\Delta (\tilde W-W) =  \partial_{\nu,\,{\rm out}}\tilde W \,H^{n-1}\restriction_{\partial U}   \quad\mbox{in } \R^n
\\
(\tilde W-W)(\infty)= 0, \\ \quad \mbox{resp. } \lim_{x\to \infty} \frac{(\tilde W-W)(x)}{-\log |x|} =  -2\pi \int_{\R^2} f \,\chi_{\R^n\setminus U}  = 2\pi \int_{\partial U}\partial_{\nu,\,{\rm out}}\tilde W \mbox{ at }\infty.
\end{cases}
\]
Therefore, $\tilde W-W$ is a single layer potential and using Theorem \ref{singlelayerthm} we obtain 
\[
\begin{split}
\| (\tilde W-W) \|_{C^{m+2,\alpha}(\overline{B_{2R}}\setminus U)} + \| (\tilde W-W) \|_{C^{m+2,\alpha}(\overline  U)}  &\le C \| \partial_{\nu,\,{\rm out}}\tilde W\|_{C^{m+1,\alpha}(\partial U)} \\ &\le C \| \tilde W \|_{C^{m+2,\alpha}(\overline{B_{2R}}\setminus U)} \le C \|f\|_{C^{m,\alpha}(\overline{B_4})}.
\end{split}
\]
Using \eqref{star2} and recalling that by definition $\tilde W\equiv 0$ in $U$ we obtain 
\[
\|W\|_{C^{m+2,\alpha}(\overline{B_{2R}} \setminus U)}\ +\|W\|_{C^{m+2,\alpha}(\overline{U})}\le C  \|f\|_{C^{m,\alpha}(\overline{B_{2R}})}.
\] 
\end{proof}

\begin{proof}[Proof of Proposition \ref{propapriori2}]
{\em Step 1.} We first prove \eqref{second2}.

Expanding \eqref{crucialrelation1} like in \eqref{crucialrelation2} but up to the next order, we find
 \begin{equation}\label{crucialrelation41}
 \partial_{\s} v^t(z,\eta^t)  = -  \partial_{\s\s} \tilde u^0(z,\eta^0) \left(\dot \eta^0  +\frac 1 2 \ddot \eta^0 t +o(t)\right)  -\frac 1 2 \partial_{\s\s\s} \tilde u^0( z,\eta^0) \left(\dot \eta^0 \right)^2 t + o(t).
 \end{equation}
 as $t= t_m \downarrow 0$.
 
Here $\eta$, $\dot \eta$ and $\ddot \eta$ are evaluated at $z$ (although we omit this in the notation) and $\partial_{\s\s} \tilde u^0(z,\eta^0)$ and $\partial_{\s\s\s} \tilde u^0( z,\eta^0)$ mean the limits from $\Omega^0$. 
To obtain the Taylor expansion up to third order of $\tilde u^0$ we are using that, by Lemma \ref{lemhigherregtildeu},  $u^0 \in C^{k+1,\alpha}\left(\overline{B_{\RR}\cap\Omega^0}\right)$ where $k\ge 2$.
Recall here that $\{u^0=0\} = \R^n\setminus \Omega^0 \subset \UU \subset B_{\RR}$.
 
Subtracting to both sides of \eqref{crucialrelation41} the quantity
\begin{equation}\label{dsv}
\partial_{\s} v(z,\eta^0) =  -  \partial_{\s\s} \tilde u^0(z,\eta^0) \dot \eta^0  \end{equation}
and dividing by $t$ we obtain 
\begin{equation}\label{second}
\frac{\partial_{\s} v^t(z,\eta^t) - \partial_{\s} v(z,\eta^0)}{t} = -\frac 1 2 \partial_{\s\s} \tilde u^0(z,\eta^0)  \ddot \eta^0-\frac 1 2 \partial_{\s\s\s} u^0( z,\eta^0) \left(\dot \eta^0 \right)^2 +o(1).
\end{equation}

Recall that by Lemma \ref{distlimit} we have $w^t\rightarrow w$ in the sense of distributions with $w$ given by \eqref{decompw}--\eqref{implicit}. 
Then, the assumption \eqref{convergencewtk} allows us to compute  the limit of the left-hand side in \eqref{second}, namely,  
\begin{equation}\label{second1}
\begin{split}
\lim_{t=t_m \downarrow 0} \frac{\partial_s v^t(z,\eta^t) -\partial_s v(z,\eta^0)}{t} 
&= \lim_{t=t_m \downarrow 0} \frac{\partial_s v(z,\eta^t) -\partial_s v (z,\eta^0)}{t} +  \frac{\partial_s v^t(z,\eta^t) -\partial_s v (z,\eta^t)}{t} 
\\
&= \partial_{ss} v(z,\eta^0)\, \dot \eta  +  \lim_{t=t_m \downarrow 0} (N(z,\eta^t) \cdot \nabla w^t (z,\eta^t) 
\\
&= \partial_{ss} v(z,\eta^0)\, \dot \eta  +  \partial_s w^t(z,\eta^0) 
\end{split}
\end{equation}
where we have used the assumption \eqref{convergencewtk}.

Taking $t= t_m \downarrow 0$ in  \eqref{second} and using \eqref{second1} we obtain
\[
 -\frac 1 2 \partial_{\s\s} \tilde u^0(z,\eta^0)  \ddot \eta^0-\frac 1 2 \partial_{\s\s\s} u^0( z,\eta^0) \left(\dot \eta^0 \right)^2  = \partial_{ss} v(z,\eta^0)\, \dot \eta  +  \partial_s w^t(z,\eta^0)  .
\]
Recalling the definition of $\Theta$  in \eqref{defTheta}  and the fact that $\partial_{\s\s} \tilde u^0 = -\Delta h^0$ on $\Gamma^0$ --- and in particular at $(z,\eta^0)$ --- we obtain  \eqref{second2}.

{\em Step 2.} We use \eqref{second2} to prove uniqueness and regularity of $\ddot\eta$.
Recall that 
\[\partial_\s w = \partial_N w = \partial_N w_{\rm sol.} +\partial_N w_{\rm sin.}+\partial_N w_{\rm dou.}+\partial_N w_{\rm imp.}\]
and while $\partial_N w_{\rm sol.}$,  $\partial_N w_{\rm sin.}$, $\partial_N w_{\rm dou.}$ depend only on ``known'' functions --- see \eqref{solid}, \eqref{single}, \eqref{double} ---   the term $\partial_N w_{\rm imp.}$ introduces a ``implicit'' dependence on $\Theta$ --- see \eqref{implicit}. We therefore need to ``solve for $\Theta$'' in \eqref{second2} in order to prove the uniqueness and regularity of its solutions $\Theta$.

For this, we write 
\[
\partial_N w_{\rm imp.} = (N\cdot \nu)\partial_\nu w_{\rm imp.}  + \big(N-(N\cdot \nu)\nu\big)\cdot \nabla w_{\rm imp.}\quad \mbox{on }\Gamma^0_{\rm out}
\]
where $\nu =\nu^0$. Recall that by a standard result on single layer potentials --- see Theorem \ref{singlelayerthm} --- we have
\begin{equation}\label{normdersingle}
(N\cdot \nu) \partial_\nu w_{\rm imp.}(x) =  \frac{\Theta(x) }{2}  + \tilde\Theta(x)\quad \mbox{on }\Gamma^0_{\rm out}.
\end{equation}
where
\begin{equation}\label{tildeTheta}
\tilde \Theta(x) :=   \int_{\Gamma_0} d\mathcal{H}^{n-1}(y) \left( -\frac{\Theta}{(N\cdot \nu) } \right) (y)\  \nu(x)\cdot \nabla P(x-y).
\end{equation} 
Note that the first term in the right-hand side of \eqref{normdersingle} is exactly the half of the first (and main) term in the left-hand side of \eqref{second2}.  Using this and denoting 
\[
\omega(x) := \big(N-(N\cdot \nu)\nu\big)(x) \quad \mbox{for $x$ on $\Gamma^0$}
\]
we obtain 
\begin{equation}\label{equationddot}
\begin{split}
\frac 12 \Theta \ =\   \frac 1 2 \partial_{\s\s\s} u^0 (\dot \eta^0 \circ\pi_1)^2  &+ \dot \eta^0\, \partial_\s N\cdot \nabla v  + \partial_\s (w_{\rm sol.}+w_{\rm sin.}+  w_{\rm dou.}) \ +
\\
 &\hspace{30mm}+ \omega\cdot \nabla w_{\rm imp.}  
+ \tilde \Theta\quad \mbox{on }\Gamma^0_{\rm out}.
\end{split}
\end{equation}

{\it Step 3.}   From \eqref{equationddot},  we  may deduce optimal regularity estimates for $\Theta$, and hence for $\ddot\eta ^0$. To do so we will bound  each of the five terms in the right-hand side of \eqref{equationddot} separately.
 
From here on, the constant   $C$ means $C=C\big(n, k,\alpha, \rrho, \|h^0\|_{C^{k+1,\alpha}(\R^n)} \big)$.
 
For the first term, we use that $h^0\in C^{k+1,\alpha}$, we obtain  that $\Gamma^0 \in C^{k, \alpha}_{\rrho/4}$ , that  $\nu^0 \in C^{k-1, \alpha}(\Gamma^0)$, and that $\eta^0\in C^{k,\alpha}(\mathcal Z)$ with estimates ---here we are using the regularity estimates on $\Gamma^0$  from Propostition \ref{blank}.
In particular, 
  \begin{equation}\label{projandnorm}
 \|\pi_1\|_{C^{k, \alpha}(\Gamma^0)} + \|\nu^0\|_{C^{k-1, \alpha}(\Gamma^0)}\le C.
 \end{equation}
Observe also that the vector field $N$ is smooth and hence $\partial_{\s\s\s}u^0$ ---the third derivative  of $u^0$ along an integral curve of $N$--- as regular as $D^3u^0$.

Therefore, 
\begin{equation}\label{1stterm} 
\begin{split}
\left\| \frac 1 2 \partial_{\s\s\s} u^0 (\dot \eta^0 \circ\pi_1)^2 \right\|_{C^{k-2,\alpha}(\Gamma^0)}
&\le 
C \big( \left\| u^0  \right\|_{C^{k+1,\alpha}(\overline{B_{\RR}\cap\Omega^0})}\left\| (\dot \eta^0\circ \pi_1)^2  \right\|_{L^\infty(\Gamma^0) }\  + 
\\
& \hspace{13mm}+\left\| u^0  \right\|_{L^\infty(B_{\RR}\cap\Omega^0)} \left\| (\dot \eta^0\circ\pi_1)^2  \right\|_{C^{k-2,\alpha}(\Gamma^0) }\big)
\\
&\le C \left\| (\dot \eta^0)^2  \right\|_{C^{k-2,\alpha}(\mathcal Z)}
\\& \le C  \boldsymbol {\mathcal Q}.
\end{split}
\end{equation}

For the second term, we use again that $N$ is smooth and recalling  the estimate  \eqref{estv} for  $v$  and the estimate $\dot\eta$ in \eqref{estdoteta0},  we obtain
\begin{equation}\label{2ndterm}
\begin{split}
\left\|  (\dot \eta^0\circ\pi_1) \,\partial_{ss}  v   \right\|_{C^{k-2,\alpha}(\Omega^0)}
&\le 
C \left( \left\|  \dot \eta^0 \right\|_{C^{k-2,\alpha}(\mathcal Z)}\left\| v  \right\|_{L^\infty (B_{\RR}\cap\Omega^0)} + \left\| \dot \eta^0   \right\|_{L^\infty(\mathcal Z)} \left\| v  \right\|_{C^{k,\alpha}(\overline{B_{\RR}\cap\Omega^0)} }\right)
\\
&\le  C \boldsymbol {\mathcal Q}.
\end{split}
\end{equation}
where we used \eqref{estdoteta0} and \eqref{estv}.

For the third term, we proceed as follows.
From Lemma \ref{potofchar} we obtain that 
\[
 \left\| \nabla w_{\rm sol}  \right\|_{C^{k-2,\alpha}(B_{\RR}\cap\Omega^0)} \le C \left\|  \Delta \ddot h^0 \right\|_{C^{k-2,\alpha}}
 \le C\boldsymbol {\mathcal Q}.
\]
Next,  since $N$ and $J$ are smooth, $\Delta h^0\in C^{k-1,\alpha}$, $\Gamma^0\in C^{k,\alpha}$, and $\nu^0\in C^{k-1,\alpha}$ we obtain by Theorem \ref{singlelayerthm} (i) that 
\[
 \left\| w_{\rm sin.}  \right\|_{C^{k-1,\alpha}(\Omega^0)} \le C\left(  \left\| (\dot \eta^0 \circ\pi_1)\Delta\dot h^0 \right\|_{C^{k-2,\alpha}(\Gamma_0)} +  \left\| (\dot \eta^0\circ \pi_1)^2  \right\|_{C^{k-2,\alpha}(\Gamma_0)} \right) 
 \le C \boldsymbol {\mathcal Q}
\]
and by Theorem \ref{singlelayerthm} (iii)
\[ \left\| w_{\rm dou.}  \right\|_{C^{k-1,\alpha}(\Omega^0)} \le C \left\| (\dot \eta^0\circ\pi_1)^2  \right\|_{C^{k-1,\alpha}(\Gamma_0)}\le C  \boldsymbol {\mathcal Q}.\]
Hence, 
\begin{equation}\label{3rdterm}
\left\| \partial_\s (w_{\rm sol.} +w_{\rm sin.}+  w_{\rm dou.})  \right\|_{C^{k-2,\alpha}(\Gamma^0)} \le C  \boldsymbol {\mathcal Q}.
\end{equation}

For the term $\omega \cdot \nabla w_{\rm imp.}$ we use that Theorem \ref{singlelayerthm} (i) yields \[ \left\| w_{\rm imp.}  \right\|_{C^{k-1,\alpha}(B_{\RR}\cap\Omega^0)} \le C \left\| \Theta \right\|_{C^{k-2,\alpha}(\Gamma_0)} \]
and thus
\begin{equation}\label{4thterm}
\left\| \omega \cdot \nabla w_{\rm imp.}  \right\|_{C^{k-2,\alpha}(\Gamma^0)} \le C\left\| \omega \right\|_{C^{k-2,\alpha}(\Gamma^0)}  \left\| \Theta  \right\|_{C^{k-2,\alpha}(\Gamma^0)} .
\end{equation}

Also,  recalling the definition of $\tilde \Theta$ in \eqref{tildeTheta} and using Theorem \ref{singlelayerthm} (iii)  we obtain %
\begin{equation}\label{5thterm}
\left\| \tilde \Theta \right\|_{C^{k-2,\alpha}(\Gamma^0)} \le C\left\| \Theta  \right\|_{C^{k-3,\alpha}(\Gamma^0)} .
\end{equation}

Inserting \eqref{1stterm}--\eqref{5thterm} into \eqref{equationddot}, we obtain 
\[
\left\| \Theta \right\|_{C^{k-2,\alpha}(\Gamma^0)} \le C\left(   \boldsymbol {\mathcal Q} + \left\| \omega \right\|_{C^{k-2,\alpha}(\Gamma^0)}  \left\| \Theta  \right\|_{C^{k-2,\alpha}(\Gamma^0)}  +\left\| \Theta  \right\|_{C^{k-3,\alpha}(\Gamma^0)}\right) .
\]

Note that we may take $ \left\| \omega  \right\|_{C^{k-2,\alpha}(\Gamma^0)}$  arbitrarily small by taking $\varepsilon_o$ in \eqref{defdelta0} small enough. Then, by a standard interpolation argument we obtain 
\begin{equation}\label{aprioritheta}
\left\|  \Theta \right\|_{C^{k-2,\alpha}(\Gamma^0)} \le C \boldsymbol {\mathcal Q} .  
\end{equation}

Finally we recall the definition of $\Theta$ in $\eqref{defTheta}$, use that $\nu^0\in C^{k-1,\alpha}$, $-\Delta h^0 \ge \rrho$, $\Delta h^0 \in C^{k-1,\alpha}$, and observe that
$\pi_0^{-1}: \mathcal Z\rightarrow \Gamma^0$ satisfies $\|\pi_0^{-1}\|_{C^{k,\alpha}(\mathcal Z)}\le C$ with $C$ universal, 
to  obtain
\begin{equation}\label{apriori2}
\left\|  \ddot\eta^0 \right\|_{C^{k-2,\alpha}(\mathcal Z)} \le C \boldsymbol {\mathcal Q}.  
\end{equation}

\end{proof}

\section{Removing the a priori assumptions}\label{sec4}
In Section \ref{sec:apriori} we assumed the existence of the limits
\begin{equation}\label{assumpC0}
 \frac{\eta^{t_m}-\eta^0}{{t_m}} \rightarrow \dot \eta^0 
\qquad\mbox{and}\qquad
2\frac{\eta^{t_m} -\eta^0- \dot \eta^0 {t_m}}{t_m^2} \rightarrow  \ddot \eta^0  \qquad\mbox{in }C^{0}(\mathcal Z)
\end{equation}
and we have shown that  $\dot \eta^0$ and $\ddot\eta^0$ must  then satisfy certain equations for which  uniqueness  and regularity estimates were proven. 

The purpose of the next section is to prove that  under our assumptions,   \eqref{assumpC0}  indeed holds for every sequence $t_m\downarrow 0$. 

\subsection{The setup}
We start by introducing a new system of coordinates in  $U_\circ\cap \overline {\Omega^0}$ that are adapted to $u^0$.

Let us define 
\begin{equation}
\sigma = \sigma(x) :=  \partial_N \tilde u_0(x) .
\end{equation}
Note that   $\sigma$ is defined in  $U_\circ\cap \overline {\Omega^0}$ and takes positive values in that neighborhood of $\Gamma^0$ if $U_\circ $ is chosen small enough. An application of the implicit function theorem gives that $(z,\sigma)$ are $C^{k,\alpha}$ coordinates in  $U_\circ\cap \overline {\Omega^0}$ (up to taking a smaller neighborhood $U_\circ$).
Indeed,  for $\nu = \nu^0$
\begin{equation}\label{changeofcords}
\frac{\partial \sigma}{\partial \s} =  \partial_{ss} \tilde u^0  = (N\cdot \nu)^2 \partial_{\nu\nu} \tilde u^0 = (N\cdot \nu)^2 \Delta \tilde u^0 = -  (N\cdot \nu)^2 \Delta h^0
\end{equation}
on $\Gamma^0_{\rm out}$ and where by assumption $-\Delta h^0\ge \rrho>0$ in a neighborhood of $\Gamma^0$.
Note in addition that the new coordinates $(z,\sigma)$ are indeed $C^{k,\alpha}$  since $\tilde u^0\in  C^{k+1,\alpha}(\overline{\Omega^0})$.
 
Let us also introduce
\[ \bar\pi_1 : U_\circ\cap\Omega^0 \rightarrow \mathcal Z\]
to be the projection defined in the coordinates $(z,\sigma)$ by 
\[ (z,\sigma)\mapsto (z,0).\]

These coordinates are clearly related to the hodograph transform of the obstacle problem introduced by Kinderlehrer and Nirenberg in \cite{KN}. Note also that  for the case of the model solution to the obstacle problem $\frac 1 2 (x_n)_+^2$ and with $N= \boldsymbol e_n$ the coordinate $\sigma$ would simply be $x_n$.

In view of Proposition \ref{blank} there exist $\lambda^t\in C^{k,\alpha}(\mathcal Z)$ such that
\begin{equation}\label{deflambda}
 \Gamma^t \ = \ \{\sigma = \lambda^t(z)\}\qquad  \mbox{ for } t\in(0,t_\circ).
 \end{equation}

In the coordinates $(z,\sigma)$ we have
\begin{equation}
\lambda^0 \equiv  0 
\end{equation}
since $\sigma = \partial_N u^0 \equiv 0$ on $\Gamma^0$.
In addition, from \eqref{crucialrelation1} and  the definition of the coordinate $\sigma$ we have 
\[
\partial_N v^t   =  -\frac{\partial_N \tilde u^0}{t} =  -\frac{\sigma}{t} = -\frac{\lambda^t}{t} \circ \bar \pi_1 \quad \mbox{on }\Gamma^t
\]
hence
\begin{equation}\label{crucialrelation}
\frac{\lambda^t}{t}(z) = -\partial_N v^t(z,\lambda^t(z)) .
\end{equation}
Indeed to prove \eqref{crucialrelation} we use \eqref{crucialrelation1} and  the definition of the coordinate $\sigma$ to obtain
\[
\partial_N v^t   =  -\frac{\partial_N \tilde u^0}{t} =  -\frac{\sigma}{t} = -\ltt\circ \bar \pi_1 \quad \mbox{on }\Gamma^t.
\]
The relation \eqref{crucialrelation}  will allow to  prove uniform $C^{k-1,\alpha}$ estimates for  
$ \frac{\lambda^t}{t}$, then leading to the existence of the limit as $t\downarrow 0$ of $\ltt$, which will be denoted $\dot \lambda^0$. Later  on, we  will prove uniform $C^{k-2,\alpha}$ estimates for 
\[\frac 1 2 \, \frac{\lambda^t -\dot \lambda^0t }{t^2} = \frac{1}{2}\, \frac{\ltt - \dot\lambda^0}{t}\]
which  will lead to the existence of its limit as $t\to 0$, denoted $\ddot \lambda^0$.
These estimates will be deduced from the equation
\begin{equation}\label{crucialrelationbis}
\frac{\ltt - \dot\lambda^0}{t}  =  - \frac{\partial_N v(z,\lambda^t(z))-\partial_N v(z,0) }{t} - \partial_N w^t(z,\lambda^t(z)),
\end{equation}
obtained from \eqref{crucialrelation} by subtracting $\dot\lambda^0(z)= -\partial_N v(z,0)$ to both sides, dividing by $t$ on both sides, and recalling that by definition $w^t= (v^t-v)/t$.

\subsection{Estimate on $\ltt$}

The goal of this subsection is to prove a  regularity result (without a priori assumptions) on $\ltt$. We state it next.
\begin{proposition}\label{lambdat/t}
For $t\in (0,t_\circ)$ we have 
\[
\left\| \ltt \right\|_{C^{k-1,\alpha}(\mathcal Z)} \le C(\boldsymbol{\mathcal C}).
\]
\end{proposition}
Before proving Proposition \ref{lambdat/t}, let us state its main corollary
\begin{corollary}\label{corvelocityexists}
There exist $\dot \eta^0$ and $\dot \lambda^0$ such that
\[
\frac{\eta^t-\eta^0}{t} \rightarrow \dot \eta^0 \qquad \mbox{and} \qquad \frac{\lambda^t}{t} \rightarrow \dot \lambda^0 \qquad\mbox{in }C^0(\mathcal Z) 
\]
as $t\downarrow 0$.
\end{corollary}
\begin{proof}
Let $t_p\downarrow 0$. 
Note that both coordinate systems $(z,s)$ and $(z,\sigma)$ are $C^{k,\alpha}$. 
Hence, the estimate $\left\| \ltt \right\|_{C^{k-1,\alpha}(\mathcal Z)} \le C$ implies 
$\left\| \frac{\eta^t-\eta^0}{t} \right\|_{C^{k-1,\alpha}(\mathcal Z)} \le C$ and by  Arzel\`a-Ascoli there is a subsequence $t_{m}$ such that 
\[ 
\frac{\eta^{t_{m}}-\eta^0}{t_{m}} \rightarrow \ell_1 \qquad \mbox{and} \qquad \frac{\lambda^{t_{m}}}{t_{m}} \rightarrow  \ell_2 \qquad\mbox{in }C^0(\mathcal Z) 
\] 
for certain limit functions $\ell_1$ and $\ell_2$ in $C^{k-1,\alpha}(\mathcal Z)$.
Applying Proposition \ref{apriori1}, we must  have  $\ell_1 =\dot \eta^0$, the function given by \eqref{eqndoteta0}. Then, either using the change of variables between $s$ and $\sigma$ or passing to the limit in \eqref{crucialrelation} we obtain 
\[
\ell_2(z) = \dot\lambda^0(z) := \partial_N v(z, \sigma=0)
\]

Therefore, we have proven that each sequence has a subsequence converging to a limit that is independent of the sequence. In other words the limits as $t\downarrow 0$ exist and are given by $\dot \eta^0$ and $\dot\lambda^0$.
\end{proof}

In view of \eqref{crucialrelation}, Proposition \ref{lambdat/t} will follow immediately  from the following
\begin{lemma}\label{controldifmain}
For $t\in (0,t_\circ)$  we have 
\[
\left\|\partial_{N} v^t (\cdot , \lambda^t(\cdot)) - \frac{1}{2} \frac{\lambda^t}{t} \right\|_{C^{k-1,\alpha}(\mathcal Z)} \le C(\boldsymbol{\mathcal C})  +  \frac{1}{100} \left\|\frac{\lambda^t}{t}\right\|_{C^{k-1,\alpha}(\mathcal Z)}.
\]
\end{lemma}

Next we state a sequence of lemmas aimed at  proving Lemma \ref{controldifmain}.
To study the regularity of $\partial_Nv^t$, let us write down (for the first time) the equation for $v^t = \frac 1 t(\tilde u^t-\tilde u^0)$ in all of $\R^n$.
We have 
\begin{equation}\label{vtall}
\begin{cases}
\Delta v^t =  -\frac{\Delta h^0}{t} \chi_{\Omega^0\setminus \Omega^t} + \Delta\delta_th^0 \chi_{\Omega^t} \quad \mbox{in }\R^n
\\
\hspace{3pt}
v^t(\infty) = \delta_t c^0, \quad \mbox{ resp. } \lim_{x\to \infty} \frac{v^t(x)}{-\log |x|} = \delta_tc^0.
\end{cases}
\end{equation}
Hense, we may decompose $v^t$ as
\[ v^t  =  v^t_{1}+ v^t_{2} + \mbox{constant},\] 
where
\begin{equation}\label{vt1}
 v^t_1(x)  := - \int_{\R^n} dy \left(\frac{\Delta h^0}{t} \chi_{\Omega^0\setminus \Omega^t} \right) (y) P(x-y)
\end{equation}
and
\begin{equation}\label{vt2}
v^t_2(x)  =   -\int_{\R^n} dy  \Delta \delta_th^0 \chi_{\Omega^t} (y) P (x-y) .
\end{equation}

To prove Lemma  \ref{controldifmain}  we will deal separately with the two contributions $\partial_Nv_1$ and $\partial_Nv_2$ to $\partial_Nv$.

Note that  $\partial_Nv_1$  is an ``approximate single layer potential''. To study its regularity we need the 
next lemma. Before giving its statement, we need to introduce some notation.

We denote
\[\bar J(z,\sigma) := | {\rm det}\,  D\, (z, \sigma)^{-1}|\] the Jacobian 
of the coordinates $(z,\sigma)$ defined by
\begin{equation}\label{jac2} \int_A f(x) \,dx = \int_{(z,\sigma)(A)} f(z,\sigma) \bar J(z,\sigma) \, dz\,d\sigma.\end{equation}
Also, for $\theta\in(0,1)$ we denote 
\[ 
\Omega^t_\theta :=Â \{ x\in U \cap \Omega^0 : \sigma(x)> \theta \lambda^t(z(x))  \}\cup (\Omega^0\setminus U),\]
\[
\Gamma^t_\theta :=\partial \Omega^t_\theta = \{\sigma = \theta\lambda^t(z)\}
\]
and $\nu^t_\theta$ the unit normal to $\Gamma^t_\theta$ towards $\Omega^t_\theta$.
Although the following lemma will be used in this subsection for $F\equiv -\Delta h^0$, we write it for general $F$ for later use.

\begin{lemma}\label{lemapproxsingle}
Let $V$ be the single layer potential
\begin{equation}\label{W1}
 V(x)  =  \int_{\R^n} dy \left(\frac{1}{t} F\,\chi_{\Omega^t\setminus \Omega^0} \right)(y) P(x-y).
\end{equation}
We may write
\begin{equation}\label{formulaV}
V=  \int_0^1 V^{\theta} \,d\theta 
\end{equation}
where 
\[
V^{\theta}  =  \int_{\Gamma^t_\theta}   \mathcal{H}^{n-1} (y)  \left( F\, \ltt \circ \bar\pi_1\, \frac{(N\cdot \nu^t_\theta)}{\partial_{ss} u^0}\right) (y) P(x-y)
\]
and  for all $\theta \in (0,1)$ we have
\begin{equation}\label{estV}
\|\nabla V^\theta\|_{C^{k-1,\alpha}(\overline{\Omega^t})} \le C(\boldsymbol{\mathcal C}) \left \|F\,   \ltt \circ \bar\pi_1\right\|_{C^{k-1,\alpha}(\Gamma^t_\theta)}.
\end{equation}
\end{lemma}

 Before giving the proof of the previous lemma let us give the analogue to Lemma \ref{LemmaJacobian} in the present context.
\begin{lemma} \label{LemmaJacobian2}
Given $f:\Gamma^t_{\theta}\rightarrow \R$ continuous we have 
\[
\int_{\Gamma^t_{\theta}} \frac{(N\cdot \nu^t_{\theta})}{\frac{\partial\sigma}{\partial s}} (x) f(x) d\mathcal{H}^{n-1}(x)  = \int_{Z} f(z,\theta\lambda^t(z)) \bar J(z,\theta\lambda^t(z)) dz. 
\]
\end{lemma}
\begin{proof}
Let us assume without loss of generality that $f$ is continuously extended in a neighborhood of $\Gamma^t_0$ contained in $U_\circ\cap \overline{\Omega^0}$.
Given $\varepsilon>0$ let $$A^\varepsilon: = \{x\in U_\circ\ :\  \theta \lambda^t(z(x)) \le \sigma(x) \le  \theta \lambda^t(z(x))+\varepsilon\}.$$
Recalling that $\frac{\partial \sigma}{\partial s} \partial_\sigma = N = \partial_s$  and that $|N| =1$,  we have
\[
\int_{\Gamma^0}  \frac{(N\cdot \nu^t_{\theta})}{\frac{\partial\sigma}{\partial s}} (x)  f(x) d\mathcal{H}^{n-1}(x)  =\lim_{\varepsilon \downarrow 0}  \frac 1\varepsilon  \int_{A^{\varepsilon}} f(x) d\mathcal{H}^n(x).
\]
On the other hand, for $$(z,\sigma)(A^\varepsilon) : = \{(z,\sigma)\in\mathcal Z\times (-\sigma_\circ,\sigma_\circ):\  \theta \lambda^t(z) \le \sigma \le \theta \lambda^t(z)+\varepsilon\}$$ we have, by definition of $\bar J$,
\[
\begin{split}
\frac 1\varepsilon \int_{A^{\varepsilon}} f(x) d\mathcal{H}^n(x)& = \frac 1\varepsilon \int_{(z,s)(A^\varepsilon)} f(z,\s) \bar J(z,\s) dz\,d\s 
\\
&=  \int_{\mathcal Z} dz\, \frac 1\varepsilon \int_{0}^{\varepsilon} d\bar s \,f(z,\theta \lambda^t(z)+\bar s) J\bar (z,\theta \lambda^t(z)+\bar s)
\\
&= \int_{\mathcal Z} f(z,\theta \lambda^t(z)) \bar J(z,\theta \lambda^t(z))\, dz + o(1)
\end{split}
\]
as $\varepsilon\downarrow 0$ and the lemma follows.
\end{proof}

\begin{proof}[Proof of Lemma \ref{lemapproxsingle}]
The key idea in the proof is to think of  an approximate  single layer potential as an average (or integral) of exact single layer potentials. 
More precisely, using \eqref{jac2} we may write
\[
\begin{split}
\int \phi\Delta V &:= \frac{1}{t} \int_{\mathcal Z} dz \int_0^{\lambda^t (z)} d\sigma  (F\phi \bar J)(z,\sigma)
\\
&=  \int_{\mathcal Z}dz \frac{1}{t}  \int_0^{1}d\theta \,\lambda^t(z) (F\phi \bar J)(z, \theta\lambda^t(z))
\\
&=   \int_0^{1} d\theta \int_{\{\sigma =\theta \lambda^t(z)\}}  \ltt (F\phi)(y)  \frac{(N\cdot \nu^t_\theta)(y)}{\frac{\partial \sigma}{\partial s}(y)} d\mathcal{H}^{n-1} (y),
\end{split}
\]\
where we used  Lemma \ref{LemmaJacobian2}. 

Recalling that $\sigma = \partial_s u^0$, this proves \eqref{formulaV}.

To prove \eqref{estV} we use that  $V^\theta$ is a single layer potential on the surface $\Gamma^t_\theta$, with charge density $(\ltt \circ\bar\pi_1) F \frac{(N\cdot \nu^t_\theta)}{u^0_{ss}}$. Note that Proposition \ref{blank} yields $\| \lambda^t \|_{C^{k,\alpha}(\mathcal Z)} \le C$ and hence $\{\sigma =\theta \lambda^t(x)\}$ is $C^{k,\alpha}$ and its normal vector $\nu^t_\theta$ is $C^{k-1,\alpha}$. Recall also that $u^0\in C^{k+1,\alpha}(\overline{\Omega^0})$ and that $u^0_{ss} \approx -(N\cdot \nu^0)^2\Delta h^0>0$ in a neighborhood of $\Gamma^0$.
Then, if 
\[F\,   \ltt \circ \bar\pi_1 \in C^{k-1,\alpha}\]
it follows from Theorem \ref{singlelayerthm} that $V^{\theta}$ is $C^{k,\alpha}(\overline{\Omega^t_\theta})$ and in  particular $V^{\theta}$ is $C^{k,\alpha}(\overline{\Omega^t})$ with the estimate \eqref{estV}.
\end{proof}

Recalling \eqref{vt1}, and using Lemma \ref{lemapproxsingle} with $F= -\Delta h^0$, we may now write 
\begin{equation}\label{decompvt1}
v_1^t(x) = \int_0^1 V^\theta(x) \,d\theta
\end{equation}
where 
\begin{equation}\label{decompvt12}
V^\theta(x) := \int_{\Gamma^t_\theta}  \left(-\Delta h^0\, \ltt \circ \bar\pi_1\, \frac{(N\cdot \nu^t_\theta)}{\partial_{ss} u^0}\right)(y) P(x-y)\,dy .
\end{equation}

The following lemma is a straightforward consequence of Theorem \ref{singlelayerthm} in the Appendix.

\begin{lemma}\label{sobolev}
Let $V^\theta$ be as in \eqref{decompvt12}.
We have
\begin{equation}\label{eqsobolev}
-\partial_{\nu^t_\theta, {\rm out}} V^\theta = \frac{1}{2}(N\cdot\nu^t_\theta) \frac{-\Delta h^0}{\partial_{ss} u^0}  \,\frac{\lambda^t}{t} \circ \bar \pi_1 + \partial_{\nu^t_\theta,0} V^\theta\quad \mbox{on  }\Gamma^t_\theta,
\end{equation} 
where
\[
\|\partial_{\nu^t_\theta,0} V^\theta \|_{C^{k-1,\alpha}(\Gamma^t_\theta)} \le C(\boldsymbol{\mathcal C})\|\ltt\|_{C^{k-2,\alpha}(\mathcal Z)}.
\]
\end{lemma}
\begin{proof}
We recall that $(N\cdot\nu^t_\theta)$, $-\Delta h^0$, $\partial_{ss} u^0> \rrho/2>0$, and   $\bar\pi_1^{-1}: \mathcal Z\rightarrow \Gamma^t_\theta$, $-\Delta h^0$ are all $C^{k-1, \alpha}$ functions. Then, the lemma follows from Theorem \ref{singlelayerthm} (ii)-(iii).
\end{proof}

The next lemma will be used to control the ``difference'' 
\[
-\partial_{N} V^\theta (z, \eta^t(z)) - \textstyle\frac{1}{2} \, \frac{\lambda^t}{t}(z) .
\]
\begin{lemma}\label{controldif}
Let $V^\theta$ as in \eqref{decompvt12}. We have
\[
\big\|-\partial_{N} V^\theta (\cdot , \eta^t(\cdot)) -\textstyle \frac{1}{2} \,\ltt \big\|_{C^{k-1,\alpha}(\mathcal Z)} \le C(\boldsymbol{\mathcal C})  + \frac{1}{100} \|\ltt\|_{C^{k-1,\alpha}(\mathcal Z)}.
\]
\end{lemma}
\begin{proof}
{\em Step 1.} We estimate the $C^{k-1,\alpha}(\mathcal Z)$ norm of 
\[
I_1(z) : =\partial_{N} V^\theta(z, \lambda^t(z)) - \partial_{N,{\rm out}} V^\theta(z,\theta \lambda^t(z)) 
\]
To do it we write this difference as 
\[I_1 = t \int_{\theta}^{1} d\bar \theta \  \partial_\sigma \partial_{N} V^s\big(x', \bar\theta \lambda^t(x')\big) \ltt(x').\]
Then,  using  Lemma \ref{lemapproxsingle} we obtain 
\begin{equation}\label{I1zzz}
\begin{split}
 \|I_1\|_{C^{k-1,\alpha}(\mathcal Z)} &\le C t \big( \|V^\theta\|_{C^{k+1,\alpha}(\Omega^t_\theta)}\|\ltt\|_{L^{\infty}(\mathcal Z)} +  \|V^\theta\|_{L^\infty(\Omega^t_\theta)}\|\ltt\|_{C^{k-1,\alpha}(\mathcal Z)} \big) 
 \\
 &\le   C t\|\ltt\|_{C^{k,\alpha}(\mathcal Z)} \|\ltt\|_{L^{\infty}(\mathcal Z)} \le  C \|\lambda^t\|_{C^{k,\alpha}(\mathcal Z)} \|\ltt\|_{L^{\infty}(\mathcal Z)} \le C,
\end{split}
\end{equation}
where $C= C(\boldsymbol{\mathcal C})$. Here we have used the fact that $ \|\ltt\|_{L^{\infty}(\mathcal Z)} \le C$, and information that follows from Proposition \ref{blank}.

{\em Step 2.} We next estimate the $C^{k-1,\alpha}(\mathcal Z)$ norm of 
\[
I_2(z):=\partial_{N, {\rm out}} V^\theta(z,\theta \lambda^t(z)) -\textstyle \frac{1}{2} \,\ltt(z).
\]
Using \eqref{eqsobolev} we have
\[
I_2(z) =  (N-\nu^t_\theta) \cdot \nabla_{\rm out} V^\theta(z,\theta \lambda^t(z)) + \frac{1}{2} \left((N\cdot\nu^t_\theta) \frac{-\Delta h^0}{\partial_{ss} u^0} -1\right) \ltt\circ \bar \pi_1 + \partial_{\nu^t_\theta,0} V^\theta.
\]
Using the estimates from Lemma \ref{sobolev} and  \ref{lemapproxsingle}  we have
\[
\|\nabla V^\theta\|_{C^{k-1,\alpha}(\Gamma^t_\theta) } \le \| V^\theta\|_{C^{k,\alpha}(\overline{\Omega^t_\theta}) } \le C \|\ltt\|_{C^{k-1,\alpha}(\mathcal Z) }.
\]
In addition,
\[
|N-\nu^t_\theta| \approx 0, \qquad (N\cdot\nu^t_\theta) \approx 1,\qquad \mbox{and}
\qquad \frac{- \Delta h^0}{\partial_{ss} u^0} \approx 1 \qquad \mbox{on }\Gamma^t_\theta
\]
for $t\in(0,t_\circ)$, where $X \approx Y$ means  that ``$X$ is arbitrarily close to $Y$''  provided that $t_\circ$ and $\varepsilon_o$ are chosen small enough depending only of $\boldsymbol{\mathcal C}$. 

Therefore, using the estimate in Lemma \ref{sobolev}  and an interpolation inequality we obtain
\begin{equation}\label{I2zzz}
\begin{split}
 \|I_2\|_{C^{k-1,\alpha}(\mathcal Z)} &\le \frac{\epsilon}{2} \|\ltt\|_{C^{k-1,\alpha}(\mathcal Z)} + C \|\ltt\|_{C^{k-2,\alpha}(\mathcal Z)} 
 \\
 &\le    \epsilon \|\ltt\|_{C^{k-1,\alpha}(\mathcal Z)} +C \|\ltt\|_{L^{\infty}(\mathcal Z)} 
 \le \epsilon \|\ltt\|_{C^{k-1,\alpha}(\mathcal Z)} +C
\end{split}
\end{equation}
where $\epsilon>0$ can be taken arbitrarily small by decreasing, if necessary, $t_\circ$ and $\varepsilon_o$.

{\em Step 3.} We conclude  by the triangle inequality that 
\[
\big\|\partial_{N} V^\theta (\cdot , \eta^t(\cdot)) -\textstyle  \frac{1}{2} \,\ltt \big\|_{C^{k-1,\alpha}(\mathcal Z)} \le \|I_1\|_{C^{k-1,\alpha}(\mathcal Z)} +\|I_2\|_{C^{k-1,\alpha}(\mathcal Z)} 
\]
and the lemma follows from \eqref{I1zzz} and \eqref{I2zzz}, setting $\epsilon = \frac{1}{100}$.
\end{proof}

The three Lemmas \ref{lemapproxsingle}, \ref{sobolev}, and \ref{controldif} will be used to treat the term $\partial_N v_1$. As a counterpart, the next lemma will be used to treat the term $\partial_N v_2$.
\begin{lemma}\label{controlvt2}
We have
\[
\| v^t_2\|_{C^{k,\alpha}(\overline{\Omega^t})}\le C(\boldsymbol{\mathcal C}).
\]
\end{lemma}
\begin{proof}
Recalling that $\Delta v^t_2  = \Delta\delta_th^0 \chi_{\Omega^t}$ and that $\Gamma^t=\partial\Omega^t$ are (uniformly)  $C^{k,\alpha}$, it  follows from Lemma \ref{potofchar}
that 
\[
 \| v^t_2\|_{C^{k,\alpha}(\overline{\Omega^t})}\le C(\boldsymbol{\mathcal C})\| \Delta\delta_th^0\|_{C^{k-2,\alpha}(\R^n)}.
\]
Using the trivial estimate
\[
\| \Delta\delta_th^0\|_{C^{k-2,\alpha}(\R^n)} \le \| h\|_{C^{k+1,\alpha}([-1,1]\times\R^n)}
\]
the lemma follows.
\end{proof}

We now can give the 

\begin{proof}[Proof of Lemma \ref{controldifmain}]
We have
\[
\partial_N v^t(z, \lambda^t(z))  = (\partial_N v^t_1 +\partial_N v^t_2)(z, \lambda^t(z))
\]
and by \eqref{decompvt1}--\eqref{decompvt12} we have 
\[
\partial_N v^t_1(z, \lambda^t(z)) =  \int_0^1  \partial_N V^\theta (z, \lambda^t(z))d\theta.
\]
Hence, by the triangle inequality, and using Lemmas \ref{controldif} and \ref{controlvt2} 
\[
\begin{split}
\left\| \partial_N v^t(\cdot , \lambda^t) -\textstyle \frac 1 2 \,\ltt(z)  \right\|_{C^{k-1,\alpha}(\mathcal Z)}  &\le
 \int_0^1 d\theta  \left\| \partial_N V^\theta (\cdot , \lambda^t)   -\textstyle  \frac 1 2 \,\ltt  \right\|_{C^{k-1,\alpha}(\mathcal Z)}  +
 \\
 &\hspace{50mm}+  \left\| \partial_N v^t_2 (\cdot , \lambda^t)  \right\|_{C^{k-1,\alpha}(\mathcal Z)}
 \\
 &\le
C  + \textstyle \frac{1}{100} \|\ltt\|_{C^{k-1,\alpha}(\mathcal Z)}  +  C\left\| \partial_N v^t_2 \right\|_{C^{k-1,\alpha}(\Gamma^t)}
\\
&\le C  +\textstyle  \frac{1}{100} \|\ltt\|_{C^{k-1,\alpha}(\mathcal Z)},
\end{split}
\] 
where $C= C(\boldsymbol{\mathcal C})$.
\end{proof}

We complete here the
\begin{proof}[Proof of Proposition \ref{lambdat/t}]
Recall \eqref{crucialrelation}, that is  
$
\ltt(z) = -\partial_N v^t(z, \lambda^t(z)) .
$
Subtracting $ \frac 1 2 \,\ltt(z)$ to both sides and using Lemma \ref{controldifmain} we obtain 
\[
\frac 1 2 \big\|\ltt \big\|_{C^{k-1,\alpha}(\mathcal Z)} \le \big\|-\partial_{N} v^t (\cdot , \lambda^t(\cdot)) - \frac{1}{2} \,\ltt \big\|_{C^{k-1,\alpha}(\mathcal Z)} \le C(\boldsymbol{\mathcal C})  +  \frac{1}{100} \|\ltt\|_{C^{k-1,\alpha}(\mathcal Z)}
\]
as desired.
\end{proof}

\subsection{Estimate on $\frac 1 t (\ltt-\dot\lambda^0)$}

The goal of this subsection is to prove  the following regularity result (without a priori assumptions)
\begin{proposition}\label{delta2lambda}
We have 
\[
\left\| \textstyle \frac {1}{t} \left( \ltt-\dot\lambda^0\right)  \right\|_{C^{k-2,\alpha}(\mathcal Z)} \le C(\boldsymbol{\mathcal C}).
\]
\end{proposition}
Before proving Proposition \ref{delta2lambda}, let us give its main corollary
\begin{corollary}\label{coraccelerationexists}
There exist $\ddot \eta^0$ and $\ddot \lambda^0$ such that
\[
2\frac{\eta^t-\eta^0- t \dot \eta^0}{t^2} \rightarrow \ddot \eta^0 \qquad \mbox{and} \qquad 2\frac{\lambda^t- t \dot \lambda^0}{t^2} \rightarrow \ddot \lambda^0 \qquad\mbox{in }C^0(\mathcal Z) 
\]
as $t\downarrow 0$.
\end{corollary}

\begin{proof}
Let $t_p\downarrow 0$. 
Note that since  both coordinate systems $(z,s)$ and $(z,\sigma)$ are $C^{k,\alpha}$ the estimate of Proposition \ref{delta2lambda} yields
\[
\left\| \frac{\eta^t-\eta^0-t \dot \eta^0}{t^2}  \right\|_{C^{k-2,\alpha}(\mathcal Z)} \le C(\boldsymbol{\mathcal C}).
\]
Hence,  by  Arzel\`a-Ascoli there is a subsequence $t_{m}$ such that 
\[ 
2\frac{\eta^{t_m}-\eta^0-   t_m \dot \eta^0}{{t_m}^2}  \rightarrow \ell_1 \qquad \mbox{and} \qquad 2\frac{\lambda^{t_m}-t_m \dot \lambda^0}{{t_m}^2} \rightarrow  \ell_2 \qquad\mbox{in }C^0(\mathcal Z) 
\] 
for certain limit functions $\ell_1$ and $\ell_2$ in $C^{k-2,\alpha}(\mathcal Z)$.

Applying Proposition \ref{propapriori2} the limit  $\ell_1 $ must be  $\ddot \eta^0$, the unique solution to \eqref{defTheta}-\eqref{second2}. Using the change of variables between $s$ and $\sigma$ we obtain that there is also a unique possible limit $\ell_2(z) = \dot\lambda^0(z)$ with is independent of the subsequence.

In other words, the limits as $t\downarrow 0$ exist and they are denoted $\ddot \eta^0$ and $\ddot\lambda^0$.
\end{proof}

In view of  \eqref{crucialrelationbis} and the regularity of $\partial_N v$, Proposition \ref{delta2lambda} will follow  from the following
\begin{lemma}\label{controldifmainw}
We have 
\[
\left\|\partial_{N} w^t (\cdot , \lambda^t(\cdot)) - \frac{1}{2} \,\frac{\ltt-\dot\lambda^0}{t} \right\|_{C^{k-2,\alpha}(\mathcal Z)} \le C(\boldsymbol{\mathcal C})  +  \frac{1}{100} \left\|\frac{\ltt-\dot\lambda^0}{t}\right\|_{C^{k-2,\alpha}(\mathcal Z)}.
\]
\end{lemma}

Let us state a sequence of lemmas which will prove Lemma \ref{controldifmainw}.
To study the regularity of $\partial_N w^t$ we will use the equation for $w^t$ in all of $\R^n$ that was obtained in \eqref{eqnwt}.

As in Step 2 of the proof of Proposition \ref{propapriori2}  we decompose 
\[w^t= w^t_1+w^t_2 + \mbox{constant}\]
where, for $n\ge3$,
\[
 w^t_1(x) =  \int  \left(\frac 1 t  \Delta \dot h^0 \chi_{\Omega^0\setminus \Omega^t} - \frac{1}{2} \Delta \delta^2_t h^0\chi_{\Omega^t} \right)(dy) P(x-y)  \quad w^t_1=0  \ 
\]
and 
\[
w^t_2 (x) =   -\int \frac 1 t \bigg(\frac{\partial_Nv}{ N\cdot \nu^0}\, \mathcal{H}^{n-1}\restriction_{\Gamma^0} - \frac 1 t \Delta h^0 \, \chi_{\Omega^0\setminus \Omega^t} \bigg) (dy) P(x-y). 
\]
Respectively, for $n=2$ we define $w^t_1$ and $w^t_2$ as the potentials of the previous Laplacians.

The analysis of the regularity in $\overline{\Omega^t}$ of $w^t_1$ is done using Lemmas \ref{lemapproxsingle} and \ref{potofchar} which straightforwardly imply
\begin{lemma}\label{lemwt1}
We have
\[
\| \nabla w^t_1\|_{C^{k-2,\alpha}(\overline{\Omega^t})}\le C(\boldsymbol{\mathcal C}).
\]
\end{lemma}
To study $w^t_2$ let us further split it as
\[w^t_2= w^t_{21}+w^t_{22} +\mbox{constant}\]
where 
\[
w^t_{21}(x) =  \int\,  \frac {1} {t \,(N\cdot \nu^0)}  \bigg( \partial_Nv + (N\cdot \nu^0)^2 \frac{\Delta h^0}{ \frac{\partial \sigma}{\partial s}} \ltt \bigg) \mathcal{H}^{n-1}\restriction_{\Gamma^0}  (dy) \,P(x-y)\]
and
\[
w^t_{22}(x) =  - \int \frac1 t \bigg( (N\cdot \nu^0)\frac{\Delta h^0}{ \frac{\partial \sigma}{\partial s}} \ltt \, \mathcal{H}^{n-1}\restriction_{\Gamma^0}  - \frac 1 t \Delta h^0 \, \chi_{\Omega^0\setminus \Omega^t} \bigg) (dy) P (x-y) .
\]

The study of $\partial_N w^t_{21}$ is done by observing that $w^t_{21}$  is a single layer potential and using Theorem \ref{singlelayerthm}. Indeed we have
\begin{lemma}\label{wt22}
We have
\[
\left\|\partial_{N} w^t_{21} (\cdot , \sigma=0) - \frac{1}{2} \,\frac{\ltt-\dot \lambda^0}{t} \right\|_{C^{k-2,\alpha}(\mathcal Z)} \le C(\boldsymbol{\mathcal C})  + \frac{1}{100} \left\|\frac{\ltt-\dot \lambda^0}{t} \right\|_{C^{k-2,\alpha}(\mathcal Z)}.
\]
\end{lemma}
\begin{proof}
Let
\[
f(x):=    \frac {1} {t \,(N\cdot \nu^0)}  \bigg( \partial_Nv + (N\cdot \nu^0)^2 \frac{\Delta h^0}{ \frac{\partial \sigma}{\partial s}} \ltt \bigg)(x)=  \frac {1} {t \,(N\cdot \nu^0)} \bigg(\dot \lambda^0 - \ltt \bigg) (x)
\] for $x\in \Gamma^0$.  Here we have used that $\partial_N v  =  -\dot \lambda \circ \bar \pi_1$ and   \eqref{changeofcords}.

On the one hand, by Theorem \ref{singlelayerthm} (iii) we have
\[
\partial_{\nu^0,\rm out} w^t_{21} = \frac 1 2 f +\partial_{\nu^0, 0} w^t_{21} \quad \mbox{on }\Gamma^0
\]
with 
\[
\| \partial_{\nu^0, 0} w^t_{21} \|_{C^{k-2,\alpha}(\Gamma^0)} \le C \| f \|_{C^{k-3,\alpha}(\Gamma^0)} 
\]
and
\[
\| w^t_{21} \|_{C^{k-1,\alpha}(\overline{\Omega^0})}\le C \| f \|_{C^{k-2,\alpha}(\Gamma^0)}, \quad \mbox{resp. } \| \nabla w^t_{21} \|_{C^{k-2,\alpha}(\overline{\Omega^0})}\le C \| f \|_{C^{k-2,\alpha}(\Gamma^0)}
\]
where $C = C(\boldsymbol{\mathcal C})$.
Therefore, using that $|N-\nu^0|\le \epsilon$ we have
\[
\begin{split}
\left\|\partial_{N} w^t_{21} - \frac 1 2  f \right\|_{C^{k-2,\alpha}(\Gamma^0)} &\le C\epsilon \left\|w^t_{21}\right\|_{C^{k-1,\alpha}(\overline{\Omega^0})}   + 
\| \partial_{\nu^0, 0} w^t_{21} \|_{C^{k-2,\alpha}(\Gamma^0)} 
\\
&\le C\epsilon \left\|f\right\|_{C^{k-2,\alpha}(\Gamma^0)}   + 
C\| f \|_{C^{k-3,\alpha}(\Gamma^0)} 
\end{split}
\]
and the lemma follows using interpolation and choosing $\epsilon$ small enough.\end{proof}

It thus remains  to study the regularity of  $w^t_{22}$,  which we treat as  an approximate double layer.

\begin{lemma}\label{reg-approx-double-layer}
We have
\[ \|\nabla w^t_{22}\|_{C^{k-2,\alpha}(\overline{\Omega^t})} \le C(\boldsymbol{\mathcal C}).\]
\end{lemma}

\begin{proof}
We will first write our approximate double layer as an average of double layers and we will then use the  regularity results for the single layers to deduce the regularity of double layers. 

Let us compute 
\[
\begin{split}
-\int \phi   \Delta w^t_{22}  &=  \int \phi (x) \frac1 t \bigg( (N\cdot \nu^0)\frac{\Delta h^0}{\frac{\partial \sigma }{\partial s}} \ltt \, \mathcal{H}^{n-1}\restriction_{\Gamma^0} - \frac 1 t \Delta h^0 \, \chi_{\Omega^0\setminus \Omega^t} \bigg)(x)\, dx
\\
&=  \frac1 t \int_{\mathcal Z}  dz\,  \ltt (z) (\bar J\Delta h^0\,\phi )  (z, 0)-  \frac{1}{t^2} \int_Z dz \int_0^{\lambda^t(z)} d\sigma    (\bar J\Delta h^0 \,\phi )(z,\sigma)
\\
&= \int_0^1 d\theta\,   \int_{\mathcal Z}  dz\, \ltt (z)    \, \frac{1}{t}\bigg(  (\bar J\Delta h^0 \phi)  (z, 0)-   (\bar J\Delta h^0 \,\phi)(z,\theta \lambda^t) \bigg)
\\
&= -\int_0^1 d\theta \int_0^\theta d\theta' \,\int_{\mathcal Z}  \big(\ltt\big)^2 (z)   dz  \partial_\sigma(\bar J\Delta h^0\,\phi)(z, \theta' \lambda^t) \bigg)
\\
&= -\int_0^1 d\theta \int_0^\theta d\theta' \,\int_{\Gamma^t_{\theta'}}\big(\ltt\big)^2\circ \bar \pi_1  \partial_{\sigma} (\bar J\Delta h^0\,\phi) \frac{(N\cdot\nu^t_{\theta'})}{\frac{\partial \sigma }{\partial s}} 
\end{split}
\]
where we have used Lemma \ref{LemmaJacobian2}.
Changing the order of integration we find
\[
\begin{split}
-\int \phi   \Delta w^t_{22}  &=   -\int_0^1 (1-\theta) d\theta  \,\int_{\Gamma^t_{\theta}}\big(\ltt\big)^2\circ \bar \pi_1  \partial_\sigma (\bar J\Delta h^0\,\phi)   \frac{(N\cdot\nu^t_{\theta})}{\frac{\partial \sigma }{\partial s}}
\\
&= -\int_0^1 (1-\theta) d\theta  \,\int_{\Gamma^t_{\theta}}\big(\ltt\big)^2\circ \bar \pi_1  \partial_s (\bar J\Delta h^0\,\phi)   \frac{(N\cdot\nu^t_{\theta})}{\big(\frac{\partial \sigma }{\partial s}\big)^2}.
\end{split}
\]
Therefore, we have
\begin{equation}\label{wt22asint}
w^t_{22}(x) = -\int_0^1 (1-\theta) d\theta  I_\theta(x)
\end{equation}
for
\[
I_\theta(x):=\,-\int_{\Gamma^t_{\theta}}d\mathcal{H}^{n-1}(y) \,\big(\ltt\big)^2\circ \bar \pi_1(y)\, \partial_N \big( (\bar J\Delta h^0)(y)\,P(x-y)\big) \frac{(N\cdot\nu^t_{\theta})}{\big(\frac{\partial \sigma }{\partial s}\big)^2}(y).
\]

Note that 
\[\begin{split}
I_\theta(x)&= I^\theta_1(x)+I^\theta_2(x)
\\
&=: \int_{\Gamma^t_{\theta}}d\mathcal{H}^{n-1}(y) \,\left(\big(\ltt\big)^2\circ \bar \pi_1\, \partial_N (\bar J\Delta h^0) \, \frac{(N\cdot\nu^t_{\theta})}{\big(\frac{\partial \sigma }{\partial s}\big)^2} \right)(y)\,P(x-y) \ 
\\
&\hspace{15mm}+ {\rm div}_x\left(  \int_{\Gamma^t_{\theta}}d\mathcal{H}^{n-1}(y) \,\left(\big(\ltt\big)^2\circ \bar \pi_1(y)\, (\bar J\Delta h^0)  \frac{(N\cdot\nu^t_{\theta})}{\big(\frac{\partial \sigma }{\partial s}\big)^2} \,N\right) (y) P(x-y)\right).
\end{split}
\]
Therefore, recalling that $\Gamma^t_\theta\in C^{k,\alpha}$, $\ltt \in C^{k-1,\alpha}(\mathcal Z)$, $\bar\pi_1\in C^{k,\alpha}(\Gamma^t_\theta)$, $\nu^t_\theta\in C^{k-1,\alpha}(\Gamma^t_\theta)$, $J\Delta h^0\in C^{k-1,\alpha}$, $\frac{\partial \sigma }{\partial s} = u^0_{ss}$ positive and $C^{k-1,\alpha}$ and using Theorem \ref{singlelayerthm} we obtain 
\[
\|\nabla I_1\|_{C^{k-2,\alpha}(\overline{\Omega^t_\theta})} + \|\nabla I_2\|_{C^{k-2,\alpha}(\overline{\Omega^t_\theta})} \le C(\boldsymbol{\mathcal C}).
\]

The estimate of the lemma then follows from \eqref{wt22asint} observing that $\overline{\Omega^t}\subset \overline{\Omega^t_\theta}$ for all $\theta\in(0,1)$.
\end{proof}

 Lemma \ref{controldifmainw} is now an immediate consequence of Lemmas  \ref{lemwt1},  \ref{wt22}, and \ref{reg-approx-double-layer}, and Proposition \ref{delta2lambda} follows.

\section{Proof of the main result}\label{sec5}

In this section we conclude the  proof of  Theorem \ref{thm1}. If one assumes that $h^{\tau+t}-h^\tau$ satisfies $\Delta(h^{\tau+t}-h^\tau)\ge 0$ and $c^{\tau+t}-c^\tau\le 0$ for $\tau,t\in(0, t_\circ)$ then Theorem~\ref{thm1} is a straightforward consequence of the results developed in Sections 2, 3, 4, and 5. Hence, the main issue that needs to be addressed is how to remove these technical sign assumptions. This is done by using a decomposition of the form
\begin{equation}\label{decomp}
h^{t}-h^0  = \xi_+^{t} + \xi_-^{t}
\end{equation}
where  $\Delta (\xi_+^{\tau+t} -\xi_+^\tau) \ge 0$ and $\lim_{x\to\infty} (\xi_+^{\tau+t} -\xi_+^\tau)  \ge 0$ and the same with $\xi_+$ replaced by $\xi_-$ and $\ge$ replaced by $\le$. 
This decomposition is defined as follows. We let 
\[ 
\phi_+(z) := 1+\frac{1+ze^z}{e^{z}+e^{-z}}\quad \mbox{ and }\quad \phi_-(z) = -1+\frac{-1 +ze^{-z}}{e^{z}+e^{-z}}
\] 
and note that
 \begin{equation}\label{ppp}
 \phi_++\phi_- = z\end{equation}
 and that  $\phi_+$ is similar to $x^+$ (the positive part) while $\phi_-$ is similar to $-x^-$ (minus the negative part) at large scales.

Let $\zeta$ be a radial smooth cutoff function with $\zeta\equiv 1$ in $B_{\RR}$ and $\zeta\equiv 0$ outside of $B_{2\RR}$.
For $t\in(-t_\circ,t_\circ)$ and $x\in \R^n$ let us define
\[\xi_+^{t}(x):=  -\int_{\R^n}P(x-y) \, t\phi_+ \bigg( \frac 1 t \Delta (h^t-h^0) \bigg) \zeta \,(y) 
\]
and 
\[\xi_-^{t}(x):=  -\int_{\R^n}P(x-y) \, t\phi_- \bigg( \frac 1 t \Delta (h^t-h^0) \bigg) \zeta \,(y) 
\]

Note that by definition we have, for $\tau$ and $t$ small,  
\begin{equation}\label{ttausmall}
\begin{split}
\Delta (\xi_+^{\tau +t}&-   \xi_+^{\tau})  =  \left( (\tau+t) \phi_+ \bigg( \frac{1}{\tau+t} \Delta (h^{\tau +t}-h^0) \bigg) - \tau\phi_+ \bigg( \frac 1 \tau \Delta (h^{\tau }-h^0) \bigg)  \right) \zeta 
\\
&= \bigg( \frac{d}{dt'} \bigg|_{t'=\tau} \bigg\{t'  \phi_+ \big(  \Delta  \delta_{t'} h^0\big)  \bigg\}\, t + O(t^{1+\alpha})\bigg)\zeta
\\
&=  \bigg(  \phi_+ \big(  \Delta \delta_\tau h^0 \big)t  +  \tau  \dot \phi_+ \big(  \Delta \delta_\tau h^0 \big)\, \frac{d}{dt'} \bigg|_{t'=\tau} \big(  \Delta  \delta_{t'} h^0\big) t    + O(t^{1+\alpha}) \bigg) \zeta
\\
&= \bigg(  \phi_+ \big(  \Delta \delta_\tau h^0 \big)t  +  \tau  \dot \phi_+ \big(  \Delta \delta_\tau h^0 \big) \frac{1}{\tau} O(\tau^\alpha)  + O(t^{1+\alpha}) \bigg) \zeta
\\ &\ge t\big(1-C\tau^\alpha -Ct^\alpha\big) \zeta \ge 0 
\end{split}
\end{equation}
where in the passage from the third to the fourth line we have used that, since $h\in C^{3,\alpha}$ ,
\[\begin{split}
 \frac{d}{dt'} \bigg|_{t'=\tau} \big(  \Delta  \delta_{t'} h^0\big) &=  \Delta \frac{d}{dt'} \bigg|_{t'=\tau} \bigg(   \frac{h^{t'}-h^0}{t'}\bigg) = \Delta \bigg( -\frac{h^{\tau}-h^0}{\tau^2} +\frac{\dot h^\tau}{\tau} \bigg)
 \\
 &=   \bigg( -\frac{\Delta \dot h^{\tau} +O(\tau^{1+\alpha})}{\tau^2}+\frac{\Delta \dot h^\tau}{\tau} \bigg) = \frac{1}{\tau} O(\tau^\alpha)
 \end{split}
\]

A similar inequality (with opposite sign) holds when $+$ is replaced by $-$.  
Moreover, by \eqref{ppp},
\[\Delta (\xi_+^{t} +\xi_-^{t})  =t\phi_+ \bigg(  \frac 1 t \Delta (h^t-h^0) \bigg)\zeta + t\phi_- \bigg( \frac 1 t \Delta (h^t-h^0) \bigg) \zeta =  \Delta (h^t-h^0) \] 
since $\Delta (h^{s+t}-h^s)=0$ outside of $B_{\RR}$ and $\zeta =1$ in $B_{\RR}$.
Therefore \eqref{decomp} follows.

Next, for $t,\bar t\in(-t_\circ,t_\circ)$ we consider the  two-parameter family of solutions to obstacle problems $u^{t,\bar t}$ defined as 
\begin{equation}\label{eq-nge32}
\min \{ -\Delta u^{t,\bar t} , u^{t,\bar t}-h^{t,\bar t}\} = 0 \quad \mbox{in }\R^n,
\qquad 
\lim_{|x|\to \infty} u^{t,\bar t}(x)  = c^{t,\bar t}  \quad \text{resp.} \lim_{|x|\to \infty}\frac{ u^{t,\bar t}(x)}{-\log |x|}  = c^{t,\bar t}  
\end{equation}
where
\[
h^{t,\bar t} :=  h^0 +\xi_+^{t} +\xi_-^{\bar t}  
\]
and 
\[
c^{t,\bar t} :=  t\phi_-\bigg( \frac 1 t (c^t-c^0) \bigg) + \bar t  \phi_+\bigg( \frac{1}{\bar t} (c^{\bar t}-c^0) \bigg).
\]
Note that
\[u^t = u^{t,t},\qquad\mbox{and}\qquad  \eta^t= \eta^{tt}.\]
Let us denote
\[ \Omega^{t,\bar t } := \{u^{t,\bar t}-h^{t,\bar t}>0\}\qquad \mbox{and} \qquad \Gamma^{t,\bar t} := \partial{\Omega^{t,\bar t}}\]
 and let $ \eta^{t,\bar t}\in C^{k,\alpha} (\mathcal Z)$ be defined by 
\begin{equation}\label{defetats}
\Gamma^{t,\bar t} \ = \ \{\s = \eta^{t,\bar t}(z)\} \subset U_\circ.
\end{equation}

In the proof  of  Theorem \ref{thm1} the following observation will be useful.
\begin{remark}\label{coneofmonotonicity}
Note that for $\boldsymbol e= (e^1,e^2) \in S^1$ making a small enough angle with $(1,0)$ a computation similar to \eqref{ttausmall} shows that
\begin{equation}\label{order}
\Delta \big(h^{t+e^1\tilde t,\  \bar t+ e^2\tilde t} -h^{t,\bar t} \big)  \ge 0  \quad \mbox{and}\quad   c^{t+e^1\tilde t,\  \bar t+ e^2\tilde t} -c^{t,\bar t}  \le 0  
\end{equation} 
for $(t,\tilde t)$ in a small neighborhood of $(0,0)$.
Thanks to this observation,  the results developed in Sections 2 to 5 can be applied to obtain, in a neighborhood of (0,0), estimates for the derivatives of  $u^{t,\bar t}$ and $\eta^{t,\bar t}$ in a cone of directions $(t,\bar t)$. 
As a consequence, we obtain estimates  for all the first and second derivatives $\partial_t$, $\partial_{\bar t}$, $\partial_{tt}$, $\partial_{\bar t \bar t}$, $\partial_{t \bar t}$  of $u^{t,\bar t}$ and $\eta^{t,\bar t}$   in a neighborhood of $(0,0)$.
In particular we obtain estimates in the direction $(1,1)$ which are equivalent to estimates for $u^t= u^{t,t}$ and $\eta^t =\eta^{t,t}$. 
\end{remark}

%

We may now give  the 
\begin{proof}[Proof of Theorem \ref{thm1}]

{\em Step 1.}  Assuming that $k\ge1$ we prove that $\eta^{t,\bar t}$ is one time  differentiable (jointly) in the two variables $(t,\bar t)$ in a neighborhood of $(0,0)$  with the estimate 
\begin{equation}\label{goalest}
\big\| \partial_{\boldsymbol e}  \eta^{t,\bar t} \big\|_{C^{k-1,\alpha}(\mathcal Z)}  \le  |\boldsymbol e| C(\boldsymbol{\mathcal C})
\end{equation}
and the formula 
\begin{equation}\label{thenicefomulaalle}
\partial_{\boldsymbol e} \eta^{t,\bar t} = \left( \frac{\partial_N (\partial_{\boldsymbol e} u^{t,\bar t}) }{(N\cdot \nu^{t,\bar t})^2 \,\Delta h^{t,\bar t}}\right) (z, \eta^{t,\bar t}(x)),
\end{equation}
which holds true for every vector ${\boldsymbol e}$ in the $(t,\bar t)$-plane.

Indeed, let $\boldsymbol e_1 =(1,0)$ and $\boldsymbol e_2$ some different unit vector making a small enough angle with $\boldsymbol e_1$ as in Remark \ref{coneofmonotonicity}.

By Remark \ref{coneofmonotonicity}, for fixed $(t,\bar t)$ in a small enough neighborhood of $(0,0)$ and for $i=1,2$, the one parameter family $(u^{t+e^1_i\tilde t,\  \bar t+ e^2_i\tilde t})_{\tilde t} $ satisfies the assumptions of Sections 2 to 4.
Applying Corollary \ref{corvelocityexists} to it, we find that 
\[\partial_{\boldsymbol e_i} \eta^{t,\bar t} := \frac{d}{d\tilde t}\bigg|_{\tilde t = 0} \eta^{t+e^1_i\tilde t,\  \bar t+ e^2_i\tilde t} \] 
exists in the sense that the limit defining this derivative exists in $C^0(\mathcal Z)$.

Then, Proposition \ref{apriori1} yields the estimate 
\[ \left\| \partial_{\boldsymbol e_i} \eta^{t,\bar t} \right\|_{C^{k-1,\alpha}(\mathcal Z)} \le C(\boldsymbol{\mathcal C}) \] 
and the formula 
\[
\partial_{\boldsymbol e_i} \eta^{t,\bar t} = \left( \frac{\partial_N \partial_{\boldsymbol e_i} (u^{t,\bar t}-h^{t,\bar t}) }{(N\cdot \nu^{t,\bar t})^2 \,\Delta h^{t,\bar t}}\right) (z, \eta^{t,\bar t}(x)).
\]

Since $\partial_t = \partial_{\boldsymbol e_1}$ and $\partial_{\bar t}$ is a linear combination of $\partial_{\boldsymbol e_i}$ we obtain that $\eta^{t,\bar t}$ is continuously differentiable (jointly) in the two variables $(t,\bar t)$ in a neighborhood of $(0,0)$  with the estimate \eqref{goalest} and formula \eqref{thenicefomulaalle}.

{\em Step 2.} Applying  \eqref{goalest} and formula \eqref{thenicefomulaalle} for $(t,\bar t)$  restricted to the ``diagonal'' $t=\bar t$ (still in a neighborhood of $(0,0)$) ---i.e. with $\boldsymbol e = (1,1)$---
we obtain that $\eta^t$ is differentiable with respect to $t$,  with the estimate
\begin{equation}\label{goalest1}
\big\|  \dot\eta^{t}\big\|_{C^{k-1,\alpha}(\mathcal Z)} \le C(\boldsymbol{\mathcal C})
\end{equation}
and the formula 
\begin{equation}\label{thenicefomulaalle1}
\dot \eta^{t} = \left( \frac{\partial_N ( \dot u^{t} -\dot h^t) }{(N\cdot \nu^{t})^2 \,\Delta h^{t}}\right) (z, \eta^{t}(x)).
\end{equation}
Note that \eqref{goalest1} and \eqref{thenicefomulaalle1} are identical to those of Proposition \ref{apriori1} but now they are valid under more general assumptions (we do not need to assume the sign condition that implies that the contact sets are ordered).

{\em Step 3}. Similarly we obtain 
\begin{equation}\label{goalest2}
\big\| \partial_{\boldsymbol e\boldsymbol e} \eta^{t,\bar t}\big\|_{C^{k-1,\alpha}(\mathcal Z)}   \le  |\boldsymbol e|^2 C(\boldsymbol{\mathcal C}).
\end{equation}

Indeed, let $\boldsymbol e_1$ and  $\boldsymbol e_2$  as in Step 1  and let  $\boldsymbol e_3$ be a third vector such that $\boldsymbol e_i$ are pairwise linearly independent and the angle of $\boldsymbol e_3$  with $(1,0)$ is small enough.

Using again Remark \ref{coneofmonotonicity}, for fixed $(t,\bar t)$ in a small enough neighborhood of $(0,0)$ and for $i=1,2,3$, the one parameter family $(u^{t+e^1_i\tilde t,\  \bar t+ e^2_i\tilde t})_{\tilde t} $ satisfies the assumptions of Sections 2 to 4.
Applying Corollary \ref{coraccelerationexists} we find that 
\[\partial_{\boldsymbol e_i \boldsymbol e_i} \eta^{t,\bar t} := \frac{d^2}{d\tilde t^2}\bigg|_{\tilde t = 0} \eta^{t+e^1_i\tilde t,\  \bar t+ e^2_i\tilde t} \] 
exists in the sense that the limit defining this derivative exists in $C^0(\mathcal Z)$.

Then, Proposition \ref{apriori1} yields the estimate 
\[ \left\| \partial_{\boldsymbol e_i \boldsymbol e_i} \eta^{t,\bar t} \right\|_{C^{k-1,\alpha}(\mathcal Z)} \le C(\boldsymbol{\mathcal C}). \]
Since for all $\boldsymbol e$ in the $(t,\bar t)$-plane $\partial_{\boldsymbol e\boldsymbol e}$ is a linear combination of  $\{\partial_{\boldsymbol e_i\boldsymbol e_i}\}_{i=1,2,3}$  we obtain that $\eta^{t,\bar t}$ is twice  differentiable (jointly) in the two variables $(t,\bar t)$ in a neighborhood of $(0,0)$  with the estimate \eqref{goalest2}.

{\em Step 4.} Applying  \eqref{goalest2} or $(t,\bar t)$  restricted to the ``diagonal'' $t=\bar t$ (still in a neighborhood of $(0,0)$) ---i.e. with $\boldsymbol e = (1,1)$---
we obtain that $\eta^t$ is  twice differentiable with respect to $t$,  with the estimate
\begin{equation}\label{goalest22}
\big\|  \ddot\eta^{t}\big\|_{C^{k-1,\alpha}(\mathcal Z)} \le C(\boldsymbol{\mathcal C}).
\end{equation}
Again note that \eqref{goalest2}  is identical to that of Proposition \ref{propapriori2} but now they are valid under more general assumptions.

{\em Step 5}. Finally, we complete the proof of Theorem \ref{thm1} by defining  the dipheomorphisms $\Psi^t$ from the coordinates $(z,\s)$ and the function $\eta^t$. 
Let $\phi \in C^\infty_c(U_\circ)$ be some function such that $\phi\equiv 1$ in a neighborhood of $\Gamma^0$.
Let us define 
\[
\Psi^t(x) = 
\begin{cases}
(z,\s)^{-1} \big( z(x), \s(x) + \eta^0(z(x)) + \phi(x) \big\{ \eta^t(z(x)) - \eta^0(z(x))\big\} \big)  \quad & x \in U_\circ
\\
x & x \in \R^n\setminus  U_\circ.
\end{cases}
\]
Sine we may take $U_\circ\subset \UU$ we have that $\Psi^t$ fixes the complement of $\UU$.
By definition of $\eta^t$ we easily show that $\Psi^t(\omega^0) = \Omega^t$ ---and thus $\Psi^t(\Gamma^0) = \Gamma^t$

It not difficult to check that \eqref{goalest1}, \eqref{thenicefomulaalle1}, and \eqref{goalest22} yield \eqref{psidot}-\eqref{diriV} and \eqref{psiddot} when rephrased in terms of $\Psi$.
On the other hand, estimate \eqref{estwt_restated} follows from the estimates for $w$ obtained in Step 3 of the proof of Proposition \ref{propapriori2}.

\end{proof}

\appendix

\section{Single layer potentials and auxiliary proofs}\label{app}
We recall here classical regularity properties and formula for the jump in the normal derivative for a single layer potential.
\begin{theorem}\label{singlelayerthm}
Let $U \subset B_R\subset \R^n$ be a domain such that $\partial U\in C^{m,\alpha}_r$ for some $r>0$, $m \in \mathbb N$ and $\alpha\in(0,1)$. 
Given $f\in C^{m-1,\alpha}(\partial U)$ let us define
\[
w(x) := \int_{\partial U} d\mathcal{H}^{n-1}(y)\, f(y) P(x-y)
\]
where $P$ is the Newtonian potential.

\vspace{3pt}
We then have: 

\vspace{3pt}
(i)  $w\in C^{0}(\R^n)$, $w\in C^{m,\alpha}(\overline{U})$ and $w\in C^{m,\alpha}(\overline{\R^n\setminus U})$ with the estimate
\[ \|w\|_{C^{m,\alpha}(\overline{U})} +  \|w\|_{C^{m,\alpha}(\overline{\R^n\setminus U})} \le C \|f\|_{C^{m-1,\alpha}(\partial U)}  \]
where $C$ depends only on $n$, $m$, $\alpha$, $r$, and $\|\partial U\|_{C^{m,\alpha}_{r}}$.

\vspace{3pt}
(ii)  Denoting $\partial_{\nu, \,{\rm out}} w$ and $\partial_{\nu, \,{\rm in}}w$ the (outward) normal derivatives of $w$ from outside and inside $U$ respectively we have, for all $x\in \partial U$,
\[
\partial_{\nu, \,{\rm out}} w(x) = \partial_{\nu, 0} w(x)  -\frac{1}{2} f(x)
\]
and
\[
\partial_{\nu, \,{\rm in}} w(x) = \partial_{\nu, 0} w(x)  +\frac{1}{2} f(x)
\]	
where
\[ \partial_{\nu, 0} w(x)  := \int_{\partial U} d\mathcal{H}^{n-1}(y)\, f(y) \, \nu(x)\cdot \nabla P(x-y) .\] 

\vspace{3pt}
(iii) The linear operator $T:f \longmapsto \partial_{\nu, 0} w$ maps continuously $C^{m-2,\alpha}(\partial U)$ to $C^{m-1,\alpha}(\partial U)$. More precisely, 
\[
\| \partial_{\nu, 0} w \|_{C^{m-1,\alpha}(\partial U) }\le C \| f\|_{C^{m-2,\alpha}(\partial U) }
\]
where $C$ depends only on $n$, $m$, $\alpha$, $R$, $r$, and $\|\partial U\|_{C^{m,\alpha}_{r}}$. In particular $T$ is compact in H\"older spaces.
\end{theorem}
For completeness we provide here a
\begin{proof}[Bibliographic references and sketch of the proof of Theorem \ref{singlelayerthm}]
Properties of single layer potentials in the spirit of (i)-(ii)-(iii) --- and related ones for double layer potentials ---  are very classical results in potential theory. They are key tools in proving the existence of solution for the Dirichlet and Neumann problems in $C^{1,\alpha}$ domains by the method of boundary potentials (by solving in H\"older spaces Fredholm integral equations on the boundary of the domain). For more information on the topic see for instance the classical books of Sobolev \cite{Sobolev} or Dautray-Lions \cite{DauLio}.

The proof of (i)-(ii) is given in  \cite[Sec. II.3]{DauLio}. The proof of (i) is given in full detail only for $m=1$ but the proof for general $m$ is similar. The result for all $m$ is stated in  \cite[p. 303]{DauLio}.

The compactness property of $T$ in  (iii) is in the core of the theory for solving the Dirichlet and Neumann problems by the method of boundary potentials. 
Indeed, by (ii), the Neumann problem $\Delta w =0$ in $U$ , $\partial_\nu =g$ on $\partial U$ is equivalent to
$Tf +\frac{1}{2}f =g$, where $f$ is the charge on the boundary. Since $T$ is compact, this equation can be solved by Fredholm's alternative\footnote{In this case the orthogonality condition of Fredholm's alternative requires $\int_{\partial U} g =0$.}; see \cite[Lectures 15-19]{Sobolev}.

Roughly speaking, the reason why $Tf$ increases by one the order of differentiability of $f$ is that the integral kernel ($x\in \partial U$)
\[k(x,y) := \nu(x)\cdot \nabla P(x-y) =  c_n \nu(x) \frac{x-y}{|x-y|^n} = O(|x-y|^{n-2})\]
as $y\to x$, $y\in \partial U$, while $\partial U$ is an $(n-1)$-dimensional surface. The extra factor $|x-y|$ comes from $\nu(x)\cdot(x-y) =O(|x-y|^2)$ since $\partial U$ is smooth enough. Thus, 
$T f$ behaves similarly to $f \longmapsto \int_{\R^d} f(y) \frac{e\cdot y}{|y|^d}\,dy$, which maps $C^{k-1,\alpha}_c(\R^d)$ to $C^{k,\alpha}(\R^d)$.

Since it is not easy to find complete references for (iii), although this type of estimates are very classical, for the sake of completeness we provide next a detailed proof of a  nearly optimal estimate like (iii) in the case $m=2$ (the proof for other $m$ is more involved but similar). 
For all the purposes of this paper the optimal estimate is not necessary --- we just state the optimal result for the convenience of the reader. In our proofs, we do not need to gain a full derivative but just  to obtain a control in a finer H\"older norm to control the corresponding term by interpolation. Let us prove that if $\partial U\in C^{2,\alpha}_r$ then, for all $\beta\in(0,1)$
\begin{equation}\label{goalapp}
\big\|Tf\big\|_{C^{0,\beta}(\partial U)}\le C \|f\|_{C^{0,\alpha}(\partial U)}
\end{equation} 
(note that the optimal estimate would be with $C^{1,\alpha}$ instead of $C^{0,\beta}$). 

As a matter of fact we will prove the stronger (and almost sharp) estimate
\begin{equation}\label{goalapp2}
\big\|Tf\big\|_{C^{0,\beta}(\partial U)}\le C \|f\|_{L^\infty(\partial U)}
\end{equation}
which clearly yields \eqref{goalapp}.

Indeed, we start by showing that
\begin{equation}\label{defk}
k(x,y) := \nu(x) \frac{x-y}{|x-y|^n}
\end{equation}
satisfies 
\begin{equation}\label{propk}
\big| k(x,y)-k(\bar x,y)\big| \le C |x-\bar x| |\xi-y|^{-n+1}
\end{equation}
where $\xi$ is a point of a curve on $\partial U$ joining $x$ and $\bar x$.

Indeed, if $\gamma\subset \partial U$ is a smooth curve joining $x$ and $\bar x$ and of length comparable to $|x-\bar x|$ we have, at $\xi =\gamma(t)$
\[
\frac{d}{dt} k(\gamma(t),y)  = \nu'(\xi) \frac{\xi-y}{|\xi-y|^n} + \nu_i(\xi)  \gamma'_j(t) \frac{|z|^2 \delta_{ij}- n z_i z_j}{|z|^{n+2}} \quad \mbox{for } z= \xi-y. 
\]
Choosing an appropriate frame we may assume that $\nu_i(\xi) = \delta_{1i}$ and  $\gamma'_j(t) = C \delta_{2j}$ ---  since the former vector is normal to $\partial U$ and the latter is tangent. Therefore 
\[
\bigg| \nu_i(\xi)  \gamma'_j(t) \frac{|z|^2 \delta_{ij}- n z_i z_j}{|z|^{n+2}} \bigg|  = C\frac{|z_1z_2|}{|z|^{n+2}} \le C\frac{|z|^2|z|}{|z|^{n+2}} \le C|z|^{1-n} = C|\xi-y|^{1-n}
\]
where we have used that the first axis is normal to $\partial U$ and hence we have $|z_1|\le |z|^2$ --- by $C^2$ regularity of $\partial U$ and recalling that $z=\xi-y$ with both $\xi$ and $y$ on $\partial U$.
Therefore, an application of the mean value theorem gives
\[
\big| k(x,y)-k(\bar x,y)\big| \le C |x-\bar x| \bigg|\frac{d}{dt} k(\gamma(t),y)    \bigg| \le C|x-\bar x| |\xi-y|^{1-n} 
\]
and proves \eqref{propk}.

Finally, recalling that 
\[  | k(x,y)| \le C |x-y|^{2-n} \qquad \mbox{and}\qquad  | k(\bar x,y)| \le C |\bar x-y|^{2-n}\]
and combining this with \eqref{propk} we obtain
\[
\big| k(x,y)-k(\bar x,y)\big| \le  C|x-\bar x|^\beta |\xi-y|^{(1-n)\beta} \big(|x-y|^{(2-n)(1-\beta)} +|\bar x-y|^{(2-n)(1-\beta)} \big) 
.\]
Therefore
 
\[
\begin{split}
|T&f(x)-Tf(\bar x)| = \bigg|\int_{\partial U}f(y)\big(k(x,y)- k(\bar x, y) \big)\, d\mathcal H^{n-1}(z)\bigg|
\\
&\le \int_{\partial U}|f(y)| \big|k(x,y)- k(\bar x, y) \big|\, d\mathcal H^{n-1}(z)
\\
&\le C\|f\|_{L^\infty(\partial U)}  \int_{\partial U} |x-\bar x|^\beta |\xi-y|^{(1-n)\beta} \big(|x-y|^{(2-n)(1-\beta)} +|\bar x-y|^{(2-n)(1-\beta)} \big)
\\
&\le C\|f\|_{L^\infty(\partial U)} |x-\bar x|^\beta,
\end{split}
\]
which proves \eqref{goalapp2}.
\end{proof}

We give here the
\begin{proof}[Sketch of the proof of Proposition \ref{blank}]
For the sake of clarity we give a proof assuming that, for $t>0$ we have $\Delta(h^t-h^0)\ge0$ and $h^t-h^0\ge0$ and thus $\Omega^t\subset \omega^0$.  
We give the proof in dimension $n=2$. The proof for $n\ge 3$ is similar; see \cite{Blank}.

{\em Step 1}.  
We show that for some $t_\circ>0$ and $C_\circ$ depending only on $\boldsymbol {\mathcal C}$ we have
\begin{equation}\label{L1est}
|\Omega^0\setminus \Omega^t| \le C(\boldsymbol{\mathcal C}) t
\end{equation}

Indeed, from \eqref{vtall} we know that (recall that $v^t := \delta_t \tilde u^0$)
\[
\begin{cases}
\Delta v^t =  -\frac{\Delta h^0}{t} \chi_{\Omega^0\setminus \Omega^t} + \Delta\delta_th^0 \chi_{\Omega^t} \quad \mbox{in }\R^2
\\
\hspace{3pt}
 \lim_{x\to \infty} \frac{v^t(x)}{-\log |x|} = \delta_tc^0.
\end{cases}
\]

Note that by \eqref{rho1}  we have $\Gamma^t \subset B_{\RR}$ for $t\in[0, t_\circ)$ where $t_\circ>0$ is a small enough constant depending only on $\boldsymbol{\mathcal C}$.
Recalling that by assumption $\Delta\delta_th^0$ is supported in $B_{\RR}$,  we have
\[
 \delta_tc^0 =  \int_{\R^2 } \Delta v^t =  \int_{\R^2 }-\frac{\Delta h^0}{t} \chi_{\Omega^0\setminus \Omega^t} + \Delta\delta_t h^0 \chi_{\Omega^t}.
\]
Therefore, since  $-\Delta h^0\ge \rrho$, we find 
\[
\frac{\rho}{t} |\Omega^0\setminus \Omega^t|  \le   |\delta_tc^0| + \int_{B_{\RR}} |\Delta\delta_t h^0|  \le C(\boldsymbol{\mathcal C}).
\]

{\em Step 2}.  We first show (i), that is we prove that for  $t_\circ$ small enough we have have 
\begin{equation} \label{CKalphaest}
\|\Gamma^t\|_{C^{k,\alpha}_{\rrho/4}} \le C_o\quad \mbox{for all }t\in [0,t_\circ).
\end{equation}
Indeed, by Step 1, $ |\Omega^0\setminus \Omega^t| \downarrow 0$ as $t\to 0$ and hence, for $t$ small enough all points of $\Gamma^t$  are regular points. More precisely, for all $p\in \Gamma^t$
\[
B_{\rrho}(p)\cap \{\tilde u^t =0\} \ge c_\circ (\boldsymbol{\mathcal C})>0
\]

Then, we  apply:
\begin{itemize}
\item[1st.]   $C^{1,\alpha}$ free boundary estimates near regular points (Caffarelli \cite{Caf1,Caf2}).
\item[2nd.]  $C^{1,\alpha} \Rightarrow C^{k,\alpha}$ estimates for obstacle $h\in C^{k+1,\alpha}$ (Kinderlehrer-Nirenberg \cite{KN}).
\end{itemize}
We thus obtain \eqref{CKalphaest}.

{\em Step 3}. From \eqref{L1est} and \eqref{CKalphaest} deduce that for $t\in (0,t_\circ)$, the Hausdorff distance between $\Gamma^{t}$ and $\Gamma^s$ satisfies 
\[ d_{\rm Hausdorff}(\Gamma^{t},\Gamma^0) \le C_o\, t.\]
\end{proof}
Finally, we give the 
\begin{proof}[Sketch of the proof of Lemma \ref{lemcharinf}]
The Lemma for $n\ge 3$ is very standard. Let us prove it in the case $n=2$.

Assume that $n=2$. We want to prove that $u^t = f_*$ where
\begin{equation}\label{tildeu2}
f_*(x):= \inf \bigg\{ f(x)  \ : \ f\in C(\R^2), \   f \ge h^t, \  \Delta f \le 0,\ \lim_{x\to \infty} \frac{f}{-\log |x|} =c^t \bigg\}
\end{equation}

The admissible class in \eqref{tildeu2} is nonempty since the function\[f_1(x) := c^t\min \{0, - \log |x|\} + C_1\] is a member, provided we take $C_1>0$ large enough that $â\log |x|+C > h^t(x)$ for all $x\in\R^2$ --- here we are using \eqref{limhinfty}. Hence, $f_*(x) \in[ h^t(x), +\infty)$ is finite for all $x$.

We now check that $u^t=f_*$ is a solution of \eqref{eq-nge3} ($n=2$). First, as an infimum of superharmonic functions, it is superharmonic.  To check that it is a subsolution of the obstacle problem, we argue by contradiction. Suppose on the contrary that there exists $r, \varepsilon,\delta > 0$ (as small as we like) and $x_\circ\in\R^2$ such that $f_* > \varepsilon+h^t$ in $B_r(x_\circ)$ and $f_*(x_\circ)> \delta+ \ave_{\partial B_r(x_\circ)} f_*$.  By changing (slightly) $x_\circ$ and making $r$ and $\delta$ smaller, if necessary, we may assume that $\delta <\ep$ and
 \[
 {\rm osc}_{B_r(x_\circ) } h^t  \le   \ep \quad \Rightarrow  \quad f_*>  \sup_{B_r(x_\circ)} h^t.
 \]  
 Let $\tilde f \in C{(B_r(x_\circ))}$ be the unique harmonic function in $B_r(x_\circ)$ with Dirichlet boundary condition  $\tilde f = \frac \delta 2 + f_*$  on $\partial B_r(x_\circ)$. Note that $\tilde f> h^t$ in  $B_r(x_\circ)$, and set 
 \[
 f(x) :=
 \begin{cases}
 f_*(x) 					&x\notin B_r(x_\circ)
 \\
 \min\{\tilde f, f_*(x) \} \quad 	& x\in B_r(x_\circ).
 \end{cases}
 \]
 Then $f$ is admissible in \eqref{tildeu2} and hence $f_*\le f$.  But then by the mean value formula for $\tilde f$ we have \[f_*(x_\circ) \le f(x_\circ)\le\tilde f(x_\circ) = \frac \delta 2 + \ave_{B_r(x_\circ)} f_* \le  \frac \delta 2 + f_*(x_\circ) - \varepsilon < f_*(x_\circ)\,, \]a contradiction.
\end{proof}

\end{document}